\documentclass[a4paper,11pt, reqno, final]{amsart} %use amsart - but that overrules geometry and messes with fancyhdr
\usepackage{ifdraft}
\usepackage[T1]{fontenc}
\usepackage[utf8]{inputenc}
\usepackage[english]{babel}
\usepackage{amsfonts,amssymb,amsmath,amsthm}
\usepackage{enumerate}
\usepackage{verbatim}
\usepackage{todonotes}
\usepackage{bbm}
\usepackage{subfigure}
\usepackage[foot]{amsaddr}

\usepackage{varioref}
\RequirePackage{doi}
\usepackage{hyperref}
\usepackage{cleveref}

\ifdraft{
\usepackage[top=4cm, bottom=2cm, left=1cm, right=4cm, marginparwidth=3cm]{geometry}
\usepackage{showlabels}
}{
\usepackage[footskip=13.0pt, top=3cm, bottom=3.5cm, left=2.5cm, right=2.5cm, marginparwidth=3cm]{geometry}
}
\newcommand{\TITLE}{Discrete Weierstrass-type representations}
\newcommand{\DATE}{\today}
  \theoremstyle{definition} %% keine Kursivschrift
  \newtheorem{defi}{Definition}[section]
	\crefname{defi}{definition}{definitions}
	
  \newtheorem{bem}[defi]{Remark}
	\crefname{bem}{remark}{remarks}
	
  \newtheorem{bsp}[defi]{Example}

	\newtheorem{examples}[defi]{Examples}
	\crefname{examples}{example}{examples}

\theoremstyle{plain} %% kursive Schrift
  \newtheorem{lem}[defi]{Lemma}
	\crefname{lem}{lemma}{lemmata}
	
  \newtheorem{satz}[defi]{Theorem}
	\crefname{satz}{theorem}{theorems}

	\crefname{cor}{corollary}{corollaries}
	
	\newtheorem{prop}[defi]{Proposition}
	\crefname{prop}{proposition}{propositions}

	\newenvironment{bspe}{\begin{examples}\leavevmode%
  \begin{enumerate}[1)]}{\end{enumerate}\end{examples}}

\newcommand{\Z}{\mathbb{Z}}
\newcommand{\R}{\mathbb{R}}
\newcommand{\C}{\mathbb{C}}
\renewcommand{\H}{\mathbb{H}}
\renewcommand{\S}{\mathbb{S}}

\newcommand{\spn}[1]{\left\langle#1\right\rangle} 
\newcommand{\Id}{\operatorname{Id}}

\renewcommand{\L}{\mathcal{L}} 

\newcommand{\Lines}{\mathcal{Z}_{\textrm{aff}}}

\renewcommand{\Re}{\operatorname{Re}}

\newcommand{\p}{\mathfrak{p}}
\renewcommand{\o}{\mathfrak{o}}
\renewcommand{\b}{\mathfrak{b}}

\renewcommand{\i}{\mathbbm{i}}

\begin{document}

\title{\TITLE}
\author{Mason Pember \and Denis Polly \and Masashi Yasumoto}


\date{\DATE}
\begin{abstract}
	Discrete Weierstrass-type representations yield a construction method in
discrete differential geometry for certain classes of discrete surfaces.
We show that the known discrete Weierstrass-type representations of
certain surface classes can be viewed as applications of the
$\Omega$-dual transform to lightlike Gauss maps in Laguerre geometry.
 From this construction, further Weierstrass-type representations arise.
As an application of the techniques we develop, we show that all
discrete linear Weingarten surfaces of Bryant or Bianchi type locally arise via
Weierstrass-type representations from discrete holomorhpic maps.
\end{abstract}

\maketitle

\section{Introduction}\label{sec:Intro}
Weierstrass-type representations play an important role in differential geometry. The classical representations, like the Weierstrass-Enneper representation for minimal surfaces in Euclidean space $\R^3$ \cite{weierstrass1866} or the representation of Bryant for constant mean curvature (cmc) 1 surfaces in hyperbolic space $\H^3$ \cite{bryant1987,umehara1992}, are convenient tools to generate interesting examples of these surface classes. Before the concept of discrete curvatures via mixed area was introduced in \cite{bobenko2010}, the existence of discrete analogues of these representations was used as justification for the definition of certain discrete surface classes. Examples of this include: 
\begin{itemize}
\item discrete isothermic and s-isothermic minimal surfaces in $\mathbb{R}^{3}$ \cite{bobenko2006,bobenko1996},
\item discrete maximal surfaces in $\mathbb{R}^{2,1}$ \cite[Thm 1.1]{yasumoto2015},
\item discrete cmc 1 surfaces in $\mathbb{H}^{3}$ \cite[Thm 4.4]{hertrich-jeromin2000}, 
\item discrete linear Weingarten surfaces of Bryant or Bianchi type (abbreviated as BrLW and BiLW, respectively) \cite{hoffmann2012,yasumoto2018}.
\end{itemize}

For all but the last example above, it has since been shown that the characterisation of these surface classes via representations coincides with their corresponding definitions via discrete curvatures (see for example \cite{bobenko2010}). Recently, this was shown for discrete flat fronts in $\mathbb{H}^{3}$, an example of a discrete BrLW surface, in \cite[Sec 3.2]{dubois2021}. 

In the smooth theory, the aforementioned surface classes belong to the class of $L$-isothermic surfaces. It is therefore natural to utilise Laguerre geometry to gain a unified geometric understanding of Weierstrass-type representations, which was done in \cite{pember2020}. To suit our goal of a unified geometric description of discrete Weierstrass-type representations\footnote{
In this paper, the term ``Weierstrass-type representation'' refers to representations that are built from discrete holomorphic functions. However, this term can also refer to other representations, such as in \cite{mueller2014} where a discrete analogue of the Kenmotsu representation is developed for discrete nonzero cmc surfaces in $\R^3$ using discrete harmonic maps.}, we seek a discrete analogue of this approach. We note that interpretations of certain discrete Weierstrass-type representations have been shown in conformal and Lie sphere geometry, such as in the works of \cite{bobenko2014,lam2020,hertrich-jeromin2021a}. 

In Section \ref{sec:Prem}, we give a brief overview of Laguerre geometry, describing the method of isotropy projection for relating points in Minkowski $4$-space $\mathbb{R}^{3,1}$ with spheres in $3$-dimensional subgeometries. Laguerre transformations can then be viewed as the affine transformations of $\mathbb{R}^{3,1}$. We give particular attention to the Laguerre transformations that are induced by hyperbolic rotations of $\mathbb{H}^{3}\subset \mathbb{R}^{3,1}$, and give a description of these in terms of $SL(2,\mathbb{C})$ matrices. 

In Section \ref{sec:Laguerre}, we describe how discrete surfaces are studied in Laguerre geometry, via discrete Legendre maps with discrete lightlike Gauss maps. We then give a discretisation of the light cone curvatures of surfaces in $\mathbb{R}^{3,1}$ developed in \cite{izumiya2009} and show how these curvatures relate to those of discrete surfaces in certain subgeometries of $\mathbb{R}^{3,1}$. We then recall the notions of discrete $\Omega$-surfaces and the $\Omega$-dual transformation developed in \cite{burstall2018}. This allows us to define the fundamental notion of discrete $L$-isothermic surfaces, being those discrete $\Omega$-surfaces that are $\Omega$-dual to their lightlike Gauss map. We then give a Laguerre geometric description of the Calapso transformation for this surface class. 

In Section \ref{sec:Weierstrass}, we apply the theory of discrete $L$-isothermic surfaces to show how various discrete Weierstrass-type representations can be seen as applications of the $\Omega$-dual transformation to a prescribed discrete lightlike Gauss map. This includes the aforementioned examples, but also gives rise to representations for new discrete surface classes such as intrinsically flat surfaces in the light cone. Moreover our results allow us to prove the existence of representations for discrete surface classes defined via discrete curvatures, in particular showing that all discrete BrLW and BiLW surfaces locally admit Weierstrass-type representations. We also show that the Calapso transformation descends to the Lawson correspondence in this setting and gives rise to a discrete analogue of the duality for cmc 1 surfaces in $\mathbb{H}^{3}$ developed in \cite{umehara1997}.

\section{Preliminaries}\label{sec:Prem}
We give a brief introduction to Laguerre geometry, the setting for our considerations: Denote by $\R^{3,1}$ the $4$-dimensional space with basis $(e_0, e_1, e_2, e_3)$, equipped with the Minkowski product $(,)$, given by
\begin{align}\label{eq:MinkProd}
	(e_0, e_0)=-1, (e_i, e_i) = 1,~i\in \{1,2,3\},~\textrm{and}~ (e_i,e_j)=0, ~\textrm{for}~i\neq j.
\end{align}
We call a vector $v\in \R^{3,1}$ \emph{spacelike}, \emph{timelike} or \emph{lightlike} if $(v,v)$ is positive, negative or equal to zero respectively.

The group of orthogonal transformations of $\R^{3,1}$, denoted $O(3,1)$, is a Lie group and its Lie algebra $\mathfrak{o}(3,1)$ is the algebra of endomorphisms that are skew-symmetric with respect to the Minkowski product. We identify $\Lambda^2 \R^{3,1}$ with $\mathfrak{o}(3,1)$ via
\begin{align*}
	(v_1\wedge v_2)w = (v_1, w)v_2 - (v_2, w)v_1.
\end{align*}
A basis of $\Lambda^2 \R^{3,1}$ (and thereby $\mathfrak{o}(3,1)$) is given by $\{e_{nm}\}_{n< m}$, with ${e_{nm}=e_n \wedge e_m}$.

In the Hermitian model of $\R^{3,1}$, points are identified with Hermitian $2\times 2$ matrices: Let $\i$ denote an imaginary unit. Then, the basis vectors of $\R^{3,1}$ become
\begin{align*}
	e_0 = \begin{pmatrix}
		1 &0 \cr 0 &1
	\end{pmatrix},~ e_1 = \begin{pmatrix}
		0 &1 \cr 1 &0
	\end{pmatrix},~ e_2 = \begin{pmatrix}
		0 &\i \cr -\i &0
	\end{pmatrix},~ e_3 = \begin{pmatrix}
		1 &0 \cr 0 &-1
	\end{pmatrix}
\end{align*}
and the inner product \eqref{eq:MinkProd} is polarization of $|v|^2= -\det v$. Each skew-symmetric endomorphism is identified with a matrix $A \in \mathfrak{sl}(2,\C)$ via
\begin{align*}
	A\cdot v = Av + vA^\ast,
\end{align*}
where $A^\ast = \overline{A}^T$. Thereby, the elements of the basis of $\mathfrak{o}(3,1)$ are identified as
\begin{align}\label{eq:HermitianSkew}
	e_{0n} = -\frac{1}{2}e_n,~ e_{12} = \frac{\i}{2}e_3,~e_{13} = -\frac{\i}{2}e_2,~e_{23} = \frac{\i}{2}e_1.
\end{align}
Furthermore, every matrix $B\in SL(2,\C)$ induces an orthogonal transformation via
\begin{align}\label{eq:SLAction}
	B \cdot v = BvB^\ast. 
\end{align}

In Laguerre geometry, isotropy projection is used to identify points in $\R^{3,1}$ with oriented spheres (see \cite[Sect 3.4]{cecil2008}). Two spheres $s_1, s_2$ are in oriented contact, if 
the difference vector $s_1 - s_2$ is lightlike. \emph{Laguerre transformations} are affine transformations of the form 
\begin{align*}
	x\mapsto cBx + b,
\end{align*}
where the \emph{linear part} $B \in O(3,1)$ is orthogonal (thus, Laguerre transformations preserve oriented contact), the \emph{translation part} $b\in\R^{3,1}$ is any vector  and $c$ is a non-zero number. We will assume that $c=1$ (thus restricting our attention to what Blaschke called the \emph{restricted Laguerre group} in \cite{blaschke1929}), because this smaller group already contains the transformation groups of the space form geometries we are interested in as subgroups (see Examples \ref{exp:LightlikeCurvature}).

A Laguerre transformation fixes orientation of spheres if its linear part is time-orientation preserving and therefore fixes
\begin{align*}
	\H^3 := \{v\in\R^{3,1}: (v,v)=-1,~v_0>0\},
\end{align*}
which is the model space of the Minkowski model of hyperbolic geometry. Thereby, isometries of $\H^3$ are the linear parts of Laguerre transformations that fix orientation of spheres. 

In the light cone model of M\"{o}bius geometry, points in the $2$-sphere $\S^2$ are represented by lines in the light cone 
\begin{align*}
	\mathcal{L} := \{v\in \R^{3,1}\colon (v,v)=0\}.
\end{align*} 
A M\"{o}bius transformation is obtained by the action of an orthogonal transformation on points in $\mathbb{P}\mathcal{L}$. Thereby we obtain a $1$-to-$1$ correspondence\footnote{For $B\in O(3,1)$, the same M\"{o}bius transformation is induced by $\pm B$, precisely one of which preserves time orientation.} between M\"{o}bius transformations of $\S^2$ and isometries of $\H^3$.

Let $B=\begin{pmatrix}
	a &b \cr c &d
\end{pmatrix}$ be an $SL(2,\C)$ matrix inducing a complex M\"{o}bius transformation
\begin{align*}
	z\mapsto \frac{az+b}{cz+d}.
\end{align*}
We aim to show that this is the M\"{o}bius transformation obtained by the action of $B$ on $\mathcal{L}$ in the Hermitian model\footnote{Note that complex M\"{o}bius transformations are even in the sense that each complex M\"{o}bius transformation is the product of an even number of reflections in a circle. Therefore, complex M\"{o}bius transformations in the light cone model correspond to $SO(3,1)$ transformations and induce isometries on $\H^3$}.

Via inverse stereographic projection, every $z\in \C$ is identified with a point in $\mathcal{L}$ by
\begin{align}\label{eq:StereoProj}
	\begin{split}
	z &\mapsto \spn{(|z|^2 +1)e_0 + (z + \overline{z})e_1 - \i(z-\overline{z})e_2 + (|z|^2-1)e_3} \\ 
	  &=\spn{\begin{pmatrix}
		|z|^2 &z \cr 
		\overline{z} &1
	\end{pmatrix}},
	\end{split}
\end{align}
where $\spn{\cdot}$ denotes the linear span of vectors. The action of $B$ on $\R^{3,1}$, $X\mapsto B\cdot X$ as defined in \eqref{eq:SLAction}, induces a M\"{o}bius transformation via
\begin{align*}
	B\cdot z = \spn{B\cdot \begin{pmatrix}
		|z|^2 &z \cr 
		\overline{z} &1
	\end{pmatrix}} = \spn{\begin{pmatrix}
		\left|\tfrac{az+b}{cz+d}\right|^2 &\tfrac{az+b}{cz+d} 
		\vspace{5pt}		\cr
		\overline{\left(\tfrac{az+b}{cz+d}\right)} &1
	\end{pmatrix}},
\end{align*}
which is indeed the action of $B$ on $\C$. 

\begin{bsp}\label{exp:Gamma}
	A hyperbolic rotation $\Gamma^{\spn{g_1}}_{\spn{g_2}}(t)$ in $\H^3$ is an orthogonal transformations for which there are two lightlike points $g_1, g_2$, such that
	\begin{align*}
		\Gamma^{\spn{g_1}}_{\spn{g_2}}(t) = \left\{ \begin{array}{ll}
			t \Id &\textrm{on}~\spn{g_1} \\
			\Id &\textrm{on}~\spn{g_1, g_2}^\perp \\ 
			t^{-1} \Id &\textrm{on}~\spn{g_2}.
		\end{array} \right.
	\end{align*}
	On any point $v \in \R^{3,1}$, we have
	\begin{align*}
		\Gamma^{\spn{g_1}}_{\spn{g_2}}(t)v = v + \frac{1-t}{(g_1, g_2)}\left(\frac{1}{t}(g_1, v)g_2 - (g_2, v)g_1\right).
	\end{align*}
	Assume, $q=\begin{pmatrix}
		1 &0 \cr 0 &0
	\end{pmatrix} \notin \spn{g_1, g_2}$. Then $g_1, g_2$ and $q$ correspond to three distinct points $z_1, z_2$ and $\infty$ in $\C\cup \{\infty\}$, respectively\footnote{Choosing $\infty$ may sound surprising, but in terms of a complex M\"{o}bius transformation, $\infty \mapsto a/c$.}. Since any complex M\"{o}bius transformation is uniquely defined by the images of three points, and 
	\begin{align*}
		\{z_1, z_2, \infty\} \mapsto \left\{z_1, z_2, \frac{z_2-tz_1}{1-t}\right\},
	\end{align*}
	we learn that $\Gamma^{\spn{g_1}}_{\spn{g_2}}(t)$ corresponds to the matrix $B\in SL(2,\C)$ with coefficients $a,b,c$ and $d$ given by its action
	\begin{align*}
		B:~z_1 \mapsto z_1, ~ z_2\mapsto z_2,~\infty \mapsto \frac{a}{c}.
	\end{align*}
	These three conditions amount to a solution matrix 
	\begin{align*}
		\tilde{B} = \begin{pmatrix}
			tz_1 -z_2 &z_1 z_2 (1-t) \cr t-1 &z_1-tz_2
		\end{pmatrix}.
	\end{align*}
	Computing the determinant of $\tilde{B}$, we learn that
	\begin{align*}
		B= \frac{1}{(z_1 - z_2)\sqrt{t}}\begin{pmatrix}
			tz_1 -z_2 &z_1 z_2 (1-t) \cr t-1 &z_1-tz_2
		\end{pmatrix},
	\end{align*}
	up to a choice of sign.
\end{bsp}

\section{Discrete Legendre immersions in Laguerre geometry}\label{sec:Laguerre}
A \emph{discrete surface} is a map $x: \Z^2 \to \R^{3,1}$. Since points in $\R^{3,1}$ represent spheres in Laguerre geometry, we also call $x$ a \emph{discrete sphere congruence}. We define 
\begin{align*}
	dx_{ij} = x_i - x_j,~\textrm{and}~ x_{ij} = \frac{x_i+x_j}{2}, 
\end{align*}
on edges and
\begin{align*}
	\delta x_{ik} = x_i - x_k,
\end{align*}
for diagonal vertices of an elementary quadrilateral $(ijkl)$ consisting of neighboring points in $\Z^2$. We call a discrete surface $x$ (cf. \cite[Def 4.6 and Rem 4.7]{burstall2020})
\begin{itemize}
	\item[-] \emph{planar}  if the vertices of each elementary quadrilateral are mapped to coplanar points in $\R^{3,1}$ and
	\item[-] \emph{circular} if it is planar and there exists $c\in \R^{3,1}$ for each elementary quadrilateral, such that
\begin{align*}
	(x-c, x-c) \equiv const.
\end{align*}
 Note that we do not ask that $c$ lies on the corresponding face of $x$. 
\end{itemize}

Let $\Lines$ denote the set of affine null-lines in $\R^{3,1}$. Each such line is spanned by two oriented spheres in contact and represents all oriented spheres in contact with the spanning spheres. Thus, every $\Lambda\in \Lines$ is identified with a contact element. In accordance with \cite{burstall2020} we pose the following definition.

\begin{defi}\label{def:LegendreNet}
	A map $L: \Z^2 \to \Lines$ is called a \emph{discrete Legendre immersion}, if for each edge $(ij)$, $L_i \cap L_j \neq \emptyset$. 
	
	A discrete sphere congruence $x:\Z^2 \to \R^{3,1}$ is \emph{enveloped by $L$}, if there is a map $G: \Z^2 \to \mathbb{P}\L$ into the projective light cone such that $L = x+G$. 
\end{defi}

The map $L$ may be interpreted as an affine null line bundle. The set of sections of $L$ is denoted by $\Gamma L$ and consists of maps $x:\Z^2 \to \R^{3,1}$ such that $x_i \in L_i$ for all $i \in \Z^2$. Thereby, $\Gamma L$ is easily seen to be the set of enveloped sphere congruences of $L$. The map $G:\Z^2 \to \mathbb{P}\L$ such that $L= x+ G$ is called the \emph{(discrete) lightlike Gau{ss} map of $L$}. 

We define the \emph{mixed area} of two edge-parallel discrete surfaces $x, x^\ast: \Z^2\to \R^{3,1}$ as a map on faces that takes values in $\wedge^2\R^{3,1}$ via
\begin{align*}
	A(x,x^\ast)_{ijkl} = \frac{1}{4}\left(\delta x_{ik}\wedge\delta x^\ast_{jl} - \delta x_{jl}\wedge \delta x^\ast_{ik}\right). 
\end{align*}
The mixed area is used to define the curvature of discrete surfaces with respect to given parallel surfaces, see \cite{bobenko2010}. For this definition to work we will pose the following assumption from now on: any discrete surface $x$ (for which we define discrete curvature) has non-parallel diagonals on each face, that is $A(x,x) \neq 0$.

To condense the curvature notions of certain subgeometries of Laguerre geometry, we introduce a curvature notion for discrete Legendre immersions; a discrete version of the light cone mean curvature given in \cite{honda2015}.

\begin{defi}\label{def:LightlikeCurvature}
	Let $x:\Z^2 \to \R^{3,1}$ be a discrete sphere congruence enveloped by the Legendre immersion ${L=x+G}$, with lightlike Gau{ss} map $G$ such that there exists a $g\in \Gamma G$ that is edge parallel to $x$. The \emph{light cone mean curvature $H_g$ of $x$ with respect to $g$} is then defined on faces via
	\begin{align*}
		A(x,g) + H_g A(x,x) = 0.
	\end{align*}
\end{defi}

\begin{bspe}\label{exp:LightlikeCurvature}
	\item Let $\p\in \R^{3,1}$ be a unit timelike vector. A circular discrete surface $f:\Z^2 \to \R^3$ may be interpreted as a circular discrete surface $x:\Z^2 \to \spn{\p}^\perp$. Since $f$ is circular, it allows for a parallel surface $n:\Z^2 \to S^2\subset \spn{\p}^\perp$ and the \emph{discrete mean curvature $H$} of $f$ with respect to its \emph{discrete Gauss map} $n$ is defined on faces via
	\begin{align}\label{eq:MeanCurvature}
		A(f,n) + H A(f,f)=0.
	\end{align}
	The net $g=\p + n$ takes values in the light cone $\L$, hence, $x$ is enveloped by $L=x+\spn{g}$ and $g$ is edge parallel to $x$. Since $\delta g = \delta n$, we have
	\begin{align*}
		H = H_g.
	\end{align*}
	The same construction of a lightlike Gau{ss} map works for discrete surfaces with timelike Gau{ss} maps in the Lorentz $3$-space $\R^{2,1}\cong \spn{\p}^\perp$ for spacelike $\p$.  We call a discrete surface in $\R^{3}$ ($\R^{2,1}$) with vanishing mean curvature \emph{minimal} (\emph{maximal}).
	\item\label{exp:lightlikeIso} For lightlike $\p$, fix $\o \in \L$ with $(\o,\p)=-1$ and $\b_1, \b_2$ as an orthonormal basis of $\spn{\o, \p}^\perp$. The subgroup of Laguerre transformations consisting of a vector in $\spn{\p}^\perp$ and an orthogonal matrix $A\in \operatorname{O}(3,1)$ with $A\p = \p$ fixes
	\begin{align*}
		I_{\p,\o} := \o + \spn{\p}^\perp,
	\end{align*}
	and acts on a point $y = \o + y_1\b_1 + y_2\b_2 + y_3\p +\in I_{\p,\o}$ as
	\begin{align*}
		y_1 &\mapsto a +y_1 \cos\phi -y_2 \sin\phi, \\
		y_2 &\mapsto b +y_1 \sin\phi +y_2 \cos\phi, \\
		y_3 &\mapsto c +c_1 y_1 +c_2 y_2 +y_3. 
	\end{align*}
	Thus, the isotropic geometry of \cite{sachs1990}, \cite{strubecker1942}, a discrete curvature theory for which has been developed in \cite{pottmann2007}, is a subgeometry of Laguerre geometry. 
	
	Let $f: \Z^2 \to I_{\p,\o}$ be a circular discrete surface, enveloped by a discrete Legendre immersion $L=f + G$ and let $g\in \Gamma G$ also take values in $I_{\p,\o}$, that is
	\begin{align*}
		g = \o + g_1\b_1 + g_2\b_2 + \frac{g_1^2+g_2^2}{2}\p.
	\end{align*}
	Since $L$ is a discrete Legendre immersion, we have $\{x_{ij} \in L_i \cap L_j\}_{ij}$ satisfying
	\begin{align*}
		x_{ij} = f_i + \lambda_ig_i = f_j + \lambda_jg_j,
	\end{align*}
	for suitable $\lambda_i, \lambda_j$. Because $f$ and $g$ take values in $I_{\p,\o}$ we have $(f,\o) = (g,\o) =-1$ and learn $\lambda_i = \lambda_j =: \lambda_{ij}$. This implies 
	\begin{align*}
		df_{ij} = -\lambda_{ij}dg_{ij},
	\end{align*}
	hence $f$ and $g$ are edge-parallel and we can define the light cone curvature $H_g$ for $f$. According to \cite[Sect 4]{pottmann2007}, in isotropic space $I_{\p,\o}$, the mean curvature of $f$ with respect to $g$ is defined by
	\begin{align*}
		A(f,g) + HA(f,f) = 0,
	\end{align*}
	which immediately implies $H=H_g$. We call a discrete surface \emph{i-minimal}, if its mean curvature vanishes.
	\item \label{exp:CMC} Let $x$ be a circular discrete surface taking values in a quadric 
		\begin{align*}
			Q_\mu:=\{x\in\R^{3,1}: (x,x)=\mu, ~\textrm{with}~\mu\neq 0\},
		\end{align*}
		which may be considered as 
		\begin{itemize}
			\item de Sitter space of curvature $\tfrac{1}{\sqrt{\mu}}$ for $\mu>0$, or
			\item two copies of a hyperbolic space form with curvature $-\tfrac{1}{\sqrt{-\mu}}$ for $\mu<0$.
		\end{itemize}
		For example, $Q_{-1}$ is two copies of hyperbolic space $\H^3$ and $Q_1$ is \emph{de Sitter space}, which we denote by $\S^{2,1}$. 
		
		The circular discrete surface $x$ allows for a $2$-parameter choice of \emph{discrete Gauss map} ${n:\Z^2 \to \R^{3,1}}$ with opposite causal character, so that $(n,x)=0$ and $L=x+\spn{x+\sqrt{|\mu|} n}$ is a discrete Legendre immersion. We call $G=\spn{x+\sqrt{|\mu|} n}$ the \emph{discrete hyperbolic Gauss map of $x$}. The \emph{discrete mean curvature $H_\mu$} of $x$ with respect to $n$ is then defined by equation \eqref{eq:MeanCurvature}. For later reference, we also define the \emph{discrete extrinsinc Gauss curvature $K$} on faces via
		\begin{align}\label{eq:GaussCurvature}
			A(n,n) - KA(f,f) = 0.
		\end{align}
		
		Because $x$ and $n$ are edge-parallel, $g=x+\sqrt{|\mu|} n$ and $x$ are edge-parallel as well and the light cone mean curvature of $x$ with respect to $g$ satisfies
		\begin{align*}
			H_g + 1 = \sqrt{|\mu|} H_\mu.
		\end{align*}
		
		Note that, conversely, for a discrete Legendre immersion $L=x+G$ enveloping a circular discrete surface $x$ in $Q_\mu$, a discrete hyperbolic Gauss map is obtained from any edge-parallel lift $g\in\Gamma G$:  
		\begin{align*}
			n = \tfrac{x}{\sqrt{|\mu|}} - \operatorname{sgn}\mu \tfrac{g}{(x, g)}.
		\end{align*}
		\item\label{exp:Lightcone} The (smooth) extrinsic geometry of spacelike surfaces in the light cone $\mathcal{L}$ has been developed in \cite{izumiya2009} (see also \cite{liu2007}), culminating in their \emph{Theorema Egregium} that a surface is intrinsically flat if and only of its light cone mean curvature vanishes, \cite[Thm 9.3]{izumiya2009}. Because of this remarkable fact (flatness is characterized by vanishing mean curvature), we call a discrete surface $x$ in the light cone with vanishing light cone mean curvature (with respect to some edge-parallel $g\in \Gamma G$) \emph{intrinsically flat (with respect to $g$)}\footnote{This terminology is solely motivated by the smooth result in \cite{izumiya2009}. Developement of a discrete curvature theory for discrete spacelike surfaces in the lightcone could be the subject of future research.}.
\end{bspe}

\subsection{Discrete L-isothermic surfaces}\label{sec:LIso}
In Laguerre geometry a \emph{discrete isothermic sphere congruence} is a discrete surface $x:\Z^2 \to \R^{3,1}$ that possesses a \emph{Christoffel dual} $x^\ast:\Z^2 \to \R^{3,1}$, which is characterised by (see \cite[Thm 4.11]{burstall2020})
\begin{itemize}
	\item $x$ and $x^\ast$ are edge parallel circular nets, and
	\item $A(x,x^\ast) =0$. 
\end{itemize}
Clearly, this relation is symmetric in $x$ and $x^\ast$, hence, $x^\ast$ is itself a discrete isothermic sphere congruence. 

Analogous to Demoulin's definition of $\Omega$-surfaces (\cite{demoulin1911}, \cite{demoulin1911a}), we give the following definition.

\begin{defi}[{\cite[Def 3.1]{burstall2018}}]\label{def:OmegaNet}
	A \emph{discrete $\Omega$-surface} is an envelope $L = x + G$ of a discrete isothermic sphere congruence.
\end{defi}

For a more general treatment of discrete $\Omega$-surfaces, see \cite{burstall2020}. 

\begin{defi}\label{def:OmegaTrafo}
	Two discrete $\Omega$-surfaces $L, L^\ast$ are called \emph{$\Omega$-dual} if there are Christoffel dual surfaces $x\in \Gamma L$ and $x^\ast \in \Gamma L^\ast$ and the discrete lightlike Gau{ss} maps coincide, i.e., $G^\ast = G$.  
\end{defi}

If $L=x+G$ is a discrete $\Omega$-surface, then we may take a Christoffel dual $x^\ast$ of $x$ and define $L^\ast = x^\ast + G$. $L^\ast$ is automatically an envelope of $x^\ast$ and, thus, an $\Omega$-dual of $L$. A special class of discrete $\Omega$-surfaces are discrete L-isothermic surfaces, which are characterised by being $\Omega$-dual to their lightlike Gau{ss} maps.

\begin{defi}\label{defi:LIsothermicNet}
	A \emph{discrete L-isothermic surface} is a discrete $\Omega$-surface $L$ that is $\Omega$-dual to its discrete lightlike Gau{ss} map $G$, that is, there exist $x\in \Gamma L$ and $g\in\Gamma G$ such that $x$ and $g$ are Christoffel dual.
\end{defi}

For a discrete L-isothermic surface, $x$ and $g$ are edge parallel with $A(x,g)=0$. Accordingly, the light cone mean curvature of $x$ with respect to $g$ vanishes on all quadrilaterals. 

\begin{bem}
	In \cite[Def 16]{bobenko2007}, the authors introduce a notion of L-isothermic discrete conical surfaces, that is, certain discrete maps into the space of hyperplanes in $\R^3$ (a discrete version of the smooth definition, where L-isothermic surfaces possess isothermic Gauss maps). These plane congruences allow for a $2$-parameter choice of an enveloped circular discrete surface  $f$ such that the resulting contact element map is a discrete L-isothermic surface. 
\end{bem}

Let $g$ be a planar discrete surface in the light cone. Then, $g$ is circular, as is any edge parallel $x$ (for details, see \cite{burstall2020}). Therefore, let $L=x + \spn{g}$ be a discrete Legendre immersion with $x$ and $g$ edge parallel and planar. Then we can define a $1$-form $\zeta$ by
\begin{align*}
	\zeta_{ij}:= g_{ij}\wedge dx_{ij} \in G_i\wedge G_j.
\end{align*}
The closedness of this $1$-form is equivalent to $A(x,g)=0$: A simple computation shows that
\begin{align*}
	\zeta_{ij}+\zeta_{jk}+\zeta_{kl}+\zeta_{li} = 2 A(x,g)_{ijkl}.
\end{align*}
We conclude: 

\begin{prop}\label{prop:closedForm}
	A discrete Legendre immersion $L$ with discrete lightlike Gau{ss} map $G$ is L-isothermic if and only if there exist discrete surfaces $g\in \Gamma G$ and $x\in\Gamma L$ such that 
	\begin{itemize}
		\item $x$ and $g$ are planar edge-parallel and
		\item the $1$-form $\zeta$, defined on edges via $\zeta_{ij}=g_{ij} \wedge dx_{ij}$, is closed. 
	\end{itemize}
\end{prop}

\subsection{The Calapso transform}\label{sec:Calapso}
The Calapso transformation of discrete $\Omega$-surfaces has been given a Lie sphere geometric treatment in \cite[Section 3]{burstall2018}, culminating in \cite[Thm and Def 3.9]{burstall2018}. However, discrete L-isothermic surfaces can be considered in Laguerre geometry, viewed as a subgeometry of Lie geometry. For self-consistency, we give a Laguerre geometric treatment of the Calapso transformation of discrete L-isothermic surfacess which coincides with the projection of the Lie geometric setting. We will start by collecting some properties of discrete isothermic surfaces in the projective light cone.

Let $g$ be a discrete isothermic surface in the light cone. Then, $G=\spn{g}$ is isothermic in the sense that there is a cross ratio factorizing edge-labelling (see \cite[Sect 4.3]{bobenko2008}): let $\phi: \Z^2 \to \C$ denote a complex valued map. On any elementary quadrilateral $(ijkl)$ we define the \emph{cross ratio of $\phi$} via
\begin{align}\label{eq:CRDef}
  [\phi_i, \phi_j, \phi_k, \phi_l] := \frac{\phi_i - \phi_j}{\phi_j - \phi_k}\frac{\phi_k - \phi_l}{\phi_l-\phi_i}. 
\end{align}
The map $G = \spn{g}$ corresponds to a complex valued map $\phi$ via stereographic projection \eqref{eq:StereoProj}. We define the \emph{cross ratio of $G$} via
\begin{align*}
  [G_i, G_j, G_k, G_l] = [\phi_i, \phi_j, \phi_k, \phi_l].
\end{align*}
Note that the identification of $G$ with $\phi$ depends on a choice of basis in $\R^{3,1}$. However, since the choice of homogeneous coordinates corresponds to a complex M\"{o}bius transformation of the complex plane, it does not affect the value of $[G_i, G_j, G_k, G_l]$. 

Then, $G$ is discrete isothermic if and only if there is an edge-function $m$ with the following properties: 
\begin{enumerate}
	\item $m$ is constant across opposite edges of faces, i.e., $m_{ij}=m_{lk}$ and $m_{il}=m_{jk}$, and
  \item $m$ factorizes the cross ratio, that is
    \begin{align*}
      cr_{ijkl} = \frac{m_{il}}{m_{ij}}<0.
    \end{align*}
\end{enumerate}
We call $m$ the \emph{cross ratio factorizing edge-labelling}\footnote{Note that many publications, such as \cite{burstall2018}, use the cross ratio factorizing function $a_{ij}=\tfrac{1}{m_{ij}}$.}. 

From now on, we assume for each elementary quadrilateral $(ijkl)\subset \Z^2$ that ${\operatorname{dim}\spn{G_i + G_j + G_k + G_l}\geq 3}$, that is, $G_i, G_j, G_k, G_l$ are not collinear.  $G$ can be characterized as isothermic by the existence of a \emph{Moutard lift} (see \cite[Thms~2.31 and 2.32]{bobenko2008}): There is a lift $\tilde{g} \in \Gamma G$ such that
\begin{align*}
	\delta \tilde{g}_{ik} ||\delta \tilde{g}_{jl}. 
\end{align*}
In our investigation we will use the following relationship to the cross ratio factorizing edge-labelling
\begin{align*}
	(\tilde{g}_i, \tilde{g}_j) = \frac{1}{m_{ij}},
\end{align*}
a proof for which may be found in \cite[Sec 3]{burstall2018}. Further, it was proved in \cite[Thm 3.6]{burstall2020} that for any closed discrete $1$-form with $\zeta_{ij} \in G_i \wedge G_j$, there is a Moutard lift $\tilde{g}$ such that
\begin{align*}
	\zeta_{ij} = -\tilde{g}_i \wedge \tilde{g}_j.
\end{align*}
As a consequence of this, we note the following Lemma.

\begin{lem}\label{lem:MoutardLiftClosed1Form}
  Let $G:\Z^2 \to \mathbb{P}\L$ be a discrete surface in the projectified light cone and $\zeta$ a discrete $1$-form with $\zeta_{ij} \in G_i \wedge G_j$. Then, $\zeta$ is closed if and only if $G$ is isothermic and for any lift $g\in \Gamma G$ of $G$
  \begin{align}\label{eq:MoutardLiftClosed1Form}
    \zeta_{ij} = \frac{1}{m_{ij}(g_i, g_j)} g_{ij} \wedge dg_{ij},
  \end{align}
  where $m$ is a cross ratio factorizing edge-labelling of $G$. 
\end{lem}

\begin{proof}
  Let $G$ be isothermic, then there is a Moutard lift ${\tilde{g}= \lambda g \in \Gamma G}$ such that $\tfrac{1}{m_{ij}} = (\tilde{g}_i, \tilde{g}_j)$. Then $\zeta_{ij} = -\tilde{g}_i \wedge \tilde{g}_j$, and closedness of $\zeta$ follows from the Moutard condition satisfied by $\tilde{g}$. 

  Conversely, if $\zeta$ is closed, there exists a Moutard lift of $G$ (which is thereby isothermic) such that
  \begin{align*}
    \zeta_{ij} = -\tilde{g}_i \wedge \tilde{g}_j = -\frac{(\tilde{g}_i,\tilde{g}_j)}{(g_i, g_j)} g_i \wedge g_j.
  \end{align*}
    and the claim follows from the fact that $\tfrac{1}{(\tilde{g}_i, \tilde{g}_j)}$ is a cross ratio factorizing edge-labelling. 
\end{proof}

$G$ comes with a family of flat connections: On the edge $(ij)$, define the map 
\begin{align*}
	\Gamma(t)_{ij}: \{j\}\times \R^{3,1} &\to \{i\}\times \R^{3,1} \\
	v&\mapsto \Gamma^{G_i}_{G_j}\left(1-\frac{t}{m_{ij}}\right)v,
\end{align*}
 that is, on each edge, $\Gamma(t)$ acts as the hyperbolic rotation (defined by $G_i$ and $G_j$) with parameter $1-\tfrac{t}{m_{ij}}$. If $t=m_{ij}$ for any edge $(ij)\subset \Z^2$, then $\Gamma(t)_{ij}$ is not well-defined; if $t=0$, then $\Gamma(0)$ is the identity. We define the set $\R^G := {\R^\times \setminus \{m_{ij}\!:~(ij) \textrm{~is edge in~}\Z^2\}}$. Then, $\Gamma(t)$ is a well-defined and non-trivial connection if $t \in \R^G$. We may write
\begin{align}\label{eq:Gamma}
	\Gamma(t)_{ij}v = v+\frac{t}{m_{ij}(g_i,g_j)}\left(\frac{1}{1-t/m_{ij}}(g_i, v)g_j - (g_j, v)g_i\right).
\end{align}
It was shown in \cite[Lemma 3.7]{burstall2018} that $\Gamma$ is flat, that is, 
\begin{align*}
	\Gamma(t)_{ij}\Gamma(t)_{jk}\Gamma(t)_{kl}\Gamma(t)_{li} = id,
\end{align*}
around each elementary quadrilateral. Via integration, we can obtain a discrete map $T(t):\Z^2 \to O(3,1)$ such that\footnote{One can think of $T(t)$ as a trivialisation of $\Gamma$ in the sense that $T(t)_i\Gamma(t)_{ij}T(t)_j^{-1} = id_{ij}$, where $id_{ij}$ denotes the trivial map $\{j\}\times \R^{3,1} \to \{i\}\times \R^{3,1}$.}
\begin{align}\label{eq:Ttrafo}
	T(t)_j = T(t)_i\Gamma(t)_{ij}.
\end{align}
Note that $T(t)$ is only unique up to a choice of initial condition.

\begin{lem}
	Let $G:\Z^2 \to \mathbb{P}\L$ be discrete isothermic with Moutard lift $\tilde{g}$. Define a map 
	$$G(t)=T(t) G:=\spn{T(t) \tilde{g}}.$$ 
	Then, $G(t)$ is also discrete isothermic with Moutard lift $\tilde{g}(t) = T(t) \tilde{g}$. 
\end{lem}

\begin{proof}
	Because $\tilde{g}$ is a Moutard lift of $G$, there is $\lambda$ such that
	\begin{align*}
		\delta \tilde{g}_{ik} = \lambda \delta \tilde{g}_{jl}.
	\end{align*}
	We compute
	\begin{align*}
		T(t)_j \tilde{g}_j - T(t)_l \tilde{g}_l &= T(t)_i\left[\Gamma(t)_{ij}\tilde{g}_j - \Gamma(t)_{il}\tilde{g}_l\right] \\
		&=T(t)_i\left[\frac{m_{il}}{m_{il}-t}(\tilde{g}_j-\tilde{g}_l) - \frac{t(m_{ij}-m_{il})}{(m_{ij}-t)(m_{il}-t)}\tilde{g}_j\right],
	\end{align*}
	and, using \eqref{eq:Gamma},
		\begin{align*}
			T(t)_i \tilde{g}_i - T(t)_k \tilde{g}_k &=T(t)_i\left[\tilde{g}_i - T(t)_i^{-1} T(t)_j \Gamma_{jk} \tilde{g}_k\right] \\
			&=T(t)_i\left[\frac{m_{jk}}{m_{jk}-t}(\tilde{g}_i-\tilde{g}_k) - \frac{t m_{jk} m_{ij}}{(m_{jk}-t)(m_{ij}-t)}(\tilde{g}_i,\tilde{g}_k)\tilde{g}_j\right].
		\end{align*}
		Now, because of 
		\begin{align*}
			(\tilde{g}_i, \tilde{g}_k) =(\tilde{g}_i, \tilde{g}_i - \delta \tilde{g}_{ik}) 
			=\lambda \left(\frac{1}{m_{il}}-\frac{1}{m_{ij}}\right) = \lambda \frac{m_{ij}-m_{il}}{m_{ij} m_{il}}, 
		\end{align*}
		we see that
		\begin{align*}
			\delta (T(t) \tilde{g})_{ik} = \lambda \delta (T(t)\tilde{g})_{jl},
		\end{align*}
		proving the lemma.
\end{proof}

\begin{defi}
	Let $L=x+G$ be a discrete L-isothermic surface, $\Gamma(t)$ the family of flat connections of its discrete lightlike Gauss map $G$, and $T(t)$ such that $T(t)_j = T(t)_i\Gamma(t)_{ij}$. The \emph{Calapso transformation} of $L$ is $L(t)=x(t) + G(t)$, where $x(t)$ and $G(t)$ are given by
	\begin{align}
		dx(t)_{ij} &:= T(t)_i x_i - T(t)_j x_j + \tfrac{t}{m_{ij}(g_i,g_j)}\left((g_i,x_i)T(t)_jg_j - (g_j, x_j)T(t)_ig_i\right)  \label{eq:CalapsoX}\\
		G(t) &:= T(t) G \label{eq:CalapsoG}.
	\end{align}
	We call $T(t)$ the \emph{lightlike part of the Calapso transformation}. 
\end{defi}

\begin{prop}
	The Calapso transformation $L(t)$ of a discrete L-isothermic surface $L$ is well-defined and L-isothermic.
\end{prop}

\begin{proof}
	We have $dx_{ij} = \lambda dg_{ij}$. Using \eqref{eq:CalapsoX} and \eqref{eq:CalapsoG} one finds $dx(t)_{ij} = \lambda d(T(t)g)_{ij}$. The Moutard equation satisfied by $\tilde{g}(t)$ is then seen to be the integrability condition of \eqref{eq:CalapsoX}. Further, $x(t)$ and $g(t):=T(t)g$ are Christoffel dual. Thus $L(t)$ is L-isothermic.  
\end{proof}

$L(t)$ itself defines a family of flat connections $\Gamma^t(s)$ (belonging to the discrete isothermic surface $G(t)$) and a map $T^t(s)$ satisfying
\begin{align*}
	T^t(s)_j = T^t(s)_i\Gamma^t(s)_{ij}.
\end{align*}
Since the Moutard lift of $G(t)$ is given by $\tilde{g}(t) = T(t)\tilde{g}$, we can compute the cross ratio factorizing edge-labelling $m(t)$ of $G(t)$ to be
\begin{align*}
	m_{ij}(t) = m_{ij}-t.
\end{align*}
Now, according to \eqref{eq:Gamma}, we have
\begin{align*}
	\Gamma^t(s)_{ij} v = v + \tfrac{s}{(m_{ij}-t)(\tilde{g}(t)_i, \tilde{g}(t)_j)}\left(
		\tfrac{m_{ij}-t}{m_{ij}-t-s}(\tilde{g}(t)_i,v)\tilde{g}(t)_j - (\tilde{g}(t)_j, v)\tilde{g}(t)_i
	\right), 
\end{align*}
and $(\tilde{g}(t)_i, \tilde{g}(t)_j) = (g_i, g_j)\frac{m_{ij}}{m_{ij}-t}$ can be used to show
\begin{align*}
	T(t)_i ^{-1}\Gamma^t_{ij}(s) T(t)_j = \Gamma(s+t)_{ij}. 
\end{align*}
Therefore, we have
\begin{align*}
	\Gamma^t_{ij}(s) = T^t(s)_i^{-1}T^t(s)_j = T(t)_iT(t+s)_i^{-1} T(t+s)_j T(t)_j^{-1},
\end{align*}
hence, with the right choice of initial condition, 
\begin{align*}
	T^t(s)T(t) = T(t+s).
\end{align*}
We proved the following:

\begin{prop}\label{lem:Additive}
	The lightlike part of the Calapso transformation is additive in the following sense: Let $L(t)$ be a Calapso transform of a discrete L-isothermic surface $L$ with parameter $t$ and let $T^t(s)$ denote the lightlike part of the Calapso transformation of $L(t)$. Then,  
	\begin{align*}
		T(s+t) = T^t(s)T(t).
	\end{align*}
	In particular,
	\begin{align*}
		T(t)^{-1} = T^t(-t).
	\end{align*}
\end{prop}

\begin{bem}
	The lightlike part of the Calapso transform is the Calapso transformation of the $\mathbb{P}\L\cong S^2$-valued map $G$. Hence Proposition \ref{lem:Additive} is a special case of \cite[Thm 3.13]{hertrich-jeromin2000} for which we gave a proof within our setup.
\end{bem}

\section{Discrete Weierstrass type represenations}\label{sec:Weierstrass}
Let $\phi: \Z^2 \to \C$ be a \emph{discrete holomorphic function}, that is, there exists an edge-labelling $m$ such that the cross ratio of $\phi$ satisfies
\begin{align*}
	[\phi_i, \phi_j, \phi_k, \phi_l]=\frac{m_{il}}{m_{ij}}<0,
\end{align*}
on any face $(ijkl)$ (see \eqref{eq:CRDef}). We can lift $\phi$ into $\L$ via inverse stereographic projection. The resulting map (see \eqref{eq:StereoProj})
\begin{align}\label{eq:phiLift}
	g =2 \begin{pmatrix}
		|\phi|^2 &\phi \\ \overline{\phi} &1
	\end{pmatrix},
\end{align}
gives rise to a discrete isothermic surface $G:=\spn{g}:\Z^2\to \mathbb{P}\L$ with the same cross ratio factorizing edge-labelling $m$, that is,
\begin{align*}
	[G_i, G_j, G_k, G_l]=\frac{m_{il}}{m_{ij}}<0. 
\end{align*}
Written in terms of the lift in \eqref{eq:phiLift}, the closed $1$-form $\zeta$ of Lemma \ref{lem:MoutardLiftClosed1Form} takes the form
\begin{align}\label{eq:zetaLift}
	\begin{split}
	\zeta_{ij} &= \Re\left\{
		\frac{1}{m_{ij}(\phi_i-\phi_j)}~\left(
			(1-\phi_i\phi_j)~e_0\wedge e_1+\i(1+\phi_i\phi_j)~e_0\wedge e_2 \right.\right. \\
			&\quad + (\phi_i + \phi_j)~e_0\wedge e_3 +\i(\phi_i+\phi_j)~e_1\wedge e_2 + (1+\phi_i\phi_j)~e_1\wedge e_3 \\ 
      &\quad\left.\left.+ \i(1-\phi_i\phi_j)~e_2\wedge e_3
		\right)
	\right\}
	\end{split}
\end{align}

\subsection{Discrete surfaces in hyperplanes of \texorpdfstring{$\R^{3,1}$}{R31}}\label{sec:Hyperplane}
Chose a non-zero $\p \in \R^{3,1}$. We then have that $\zeta\p$ is a closed $1$-form with values in $\R^{3,1}$ and we may (locally) integrate to obtain a discrete surface $x:\Z^2\to\R^{3,1}$ such that $dx_{ij}=-\zeta_{ij}\p$ (see \cite[Prop. 2.17]{burstall2020}). Then, $dx\perp \p$ because $\zeta$ is skew-symmetric. Hence, $x$ takes values in an affine hyperplane with normal $\p$. If at a vertex $i\in \Z^2$ $\p \perp G_i$, then $\zeta_{ij}\p \in G_i$ and $\zeta_{il}\p\in G_i$ are parallel, hence the image of the quadrilateral $(ijkl)$ under $x$ degenerates. 

Thus, for non-degeneracy, we assume $\p\not\perp G_i$ and we may choose $g^n\in \Gamma G$ such that ${(g^n,\p)=-1}$. Then 
\begin{align*}
	dx_{ij} = -\zeta_{ij}\p = \frac{1}{m_{ij}(g^n_i,g^n_j)}dg^n_{ij}, 
\end{align*}
hence, $x$ is edge-parallel to $g^n$. Moreover, we may now write $\zeta_{ij} = g^n_{ij} \wedge dx_{ij}$. By Proposition \ref{prop:closedForm}  $L=x+G$ is a discrete L-isothermic surface and $x$ and $g^n$ are Christoffel dual. In particular, the light cone mean curvature of $x$ with respect to $g^n$ vanishes, and we learn the following from the examples given in \ref{exp:LightlikeCurvature}: 

\begin{itemize}
	\item If $\p$ is timelike, $x$ is discrete minimal in  Euclidean $3$-space.
	\item If $\p$ is spacelike, $x$ is discrete maximal in  Lorentz $3$-space.
	\item If $\p$ is lightlike, $x$ is discrete i-minimal in  isotropic $3$-space.
\end{itemize}

\begin{figure}%
	\centering
	\subfigure[][The discrete power function $z\mapsto z^{2/3}$ used to create the discrete surfaces in \ref{fig:ZMCPower}]{%
		\label{fig:holoPower}
		\includegraphics[width=0.25\textwidth]{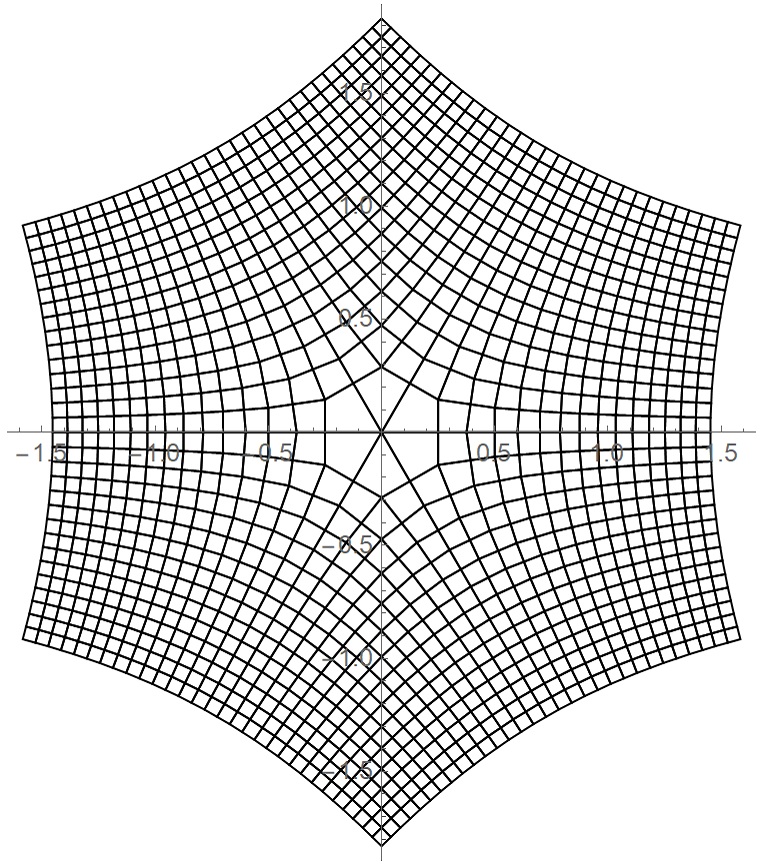}%
	}%
  \hspace{0.8cm}
    \subfigure[][Using the discrete holomorphic power function in \ref{fig:holoPower} and \eqref{eq:WeierFormula} we obtain a discrete minimal surface in Euclidean space ($\mu = 1$) on the left, a discrete maximal surface in Lorentz space ($\mu = -1$) on the right and a discrete i-minimal surface ($\mu = 0$) in the middle]{%
    \label{fig:ZMCPower}
    \includegraphics[width=0.6\textwidth]{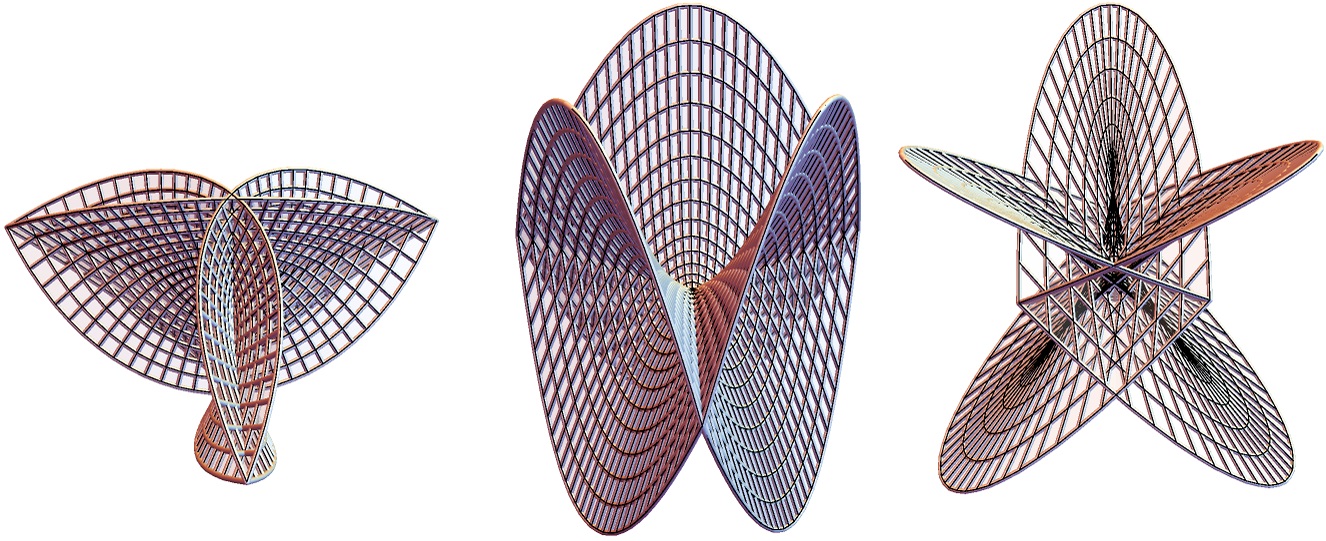}%
	}%
  \\
    \subfigure[][The discrete power function $z\mapsto z^{2/3}$, translated and used to create the discrete surfaces in \ref{fig:ZMCTransPower}]{%
		\label{fig:holoTransPower}%
		\includegraphics[width=0.25\textwidth]{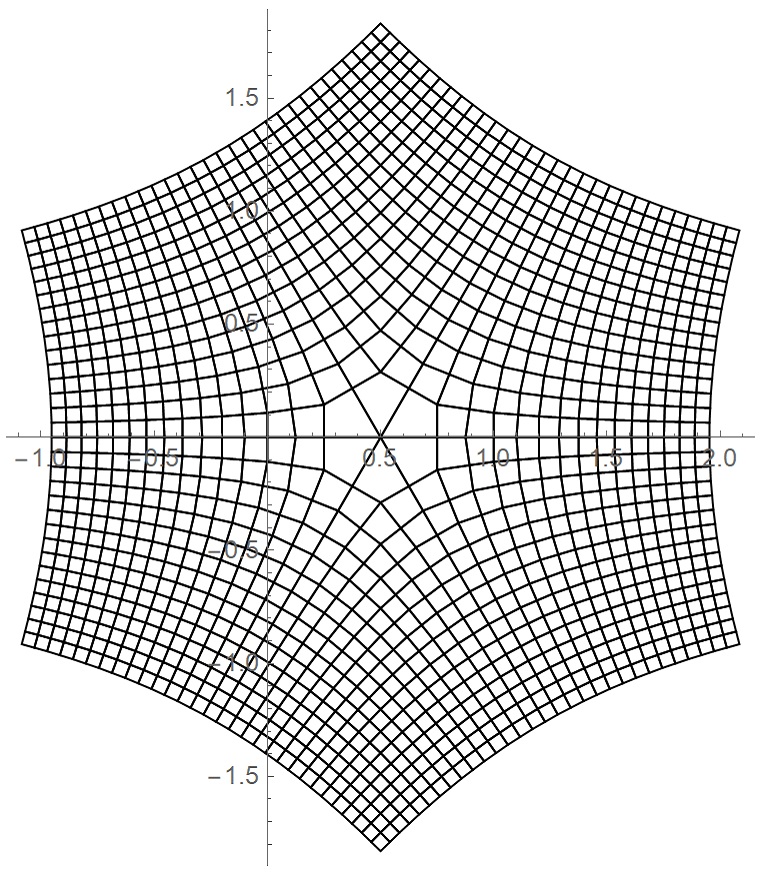}%
	}%
  \hspace{0.8cm}
    \subfigure[][Using the translated discrete holomorphic power function in \ref{fig:holoTransPower} and \eqref{eq:WeierFormula} we obtain a discrete minimal surface in Euclidean space ($\mu = 1$) on the left, a discrete maximal surface in Lorentz space ($\mu = -1$) on the right and a discrete i-minimal surface ($\mu = 0$) in the middle]{%
    \label{fig:ZMCTransPower}
    \includegraphics[width=0.6\textwidth]{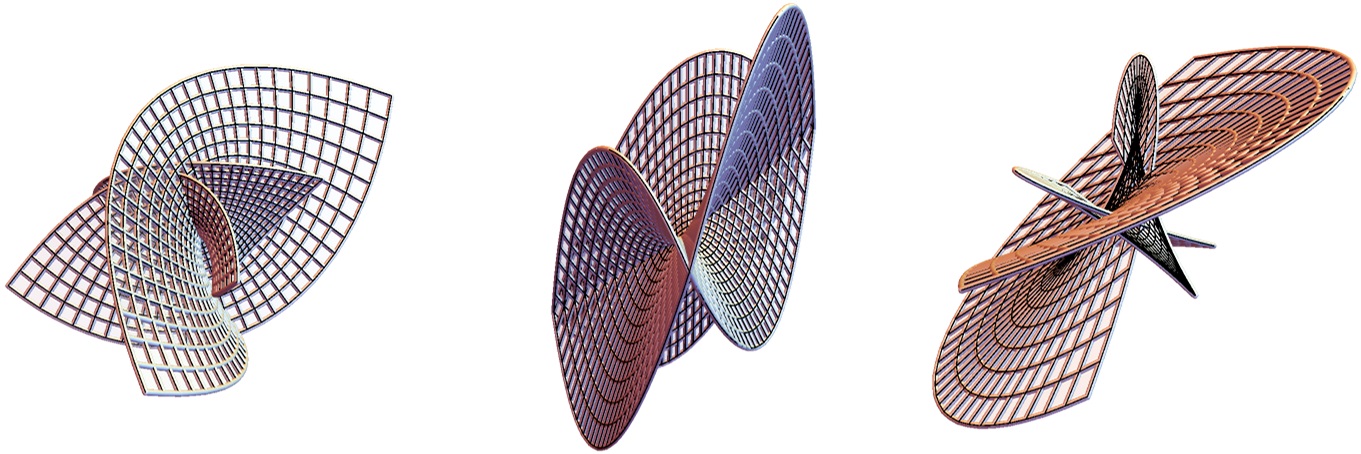}%
	}%
  \\
    \subfigure[][The discrete ex\-po\-nen\-tial map,  translated and used to create the discrete surfaces in \ref{fig:ZMCTransExp}]{%
		\label{fig:holoTransExp}%
		\includegraphics[width=0.25\textwidth]{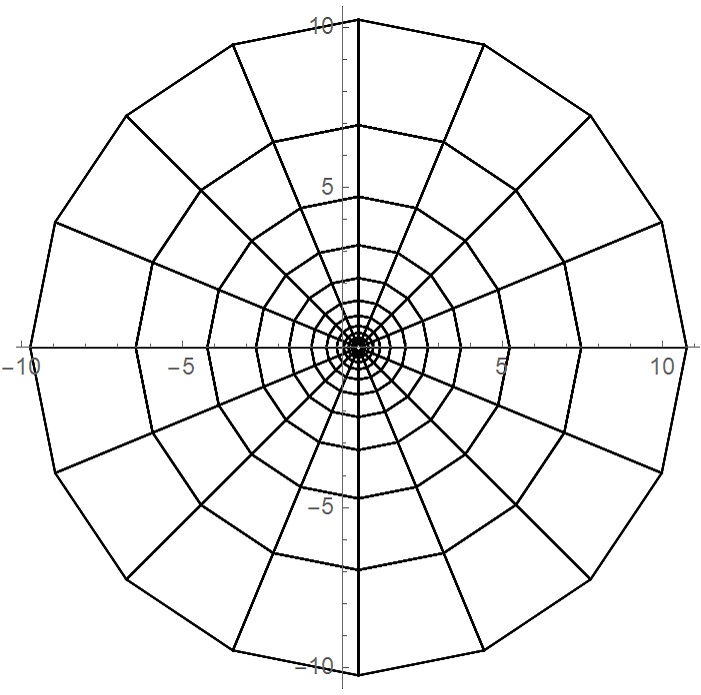}%
	}% 
  \hspace{0.8cm}
    \subfigure[][Using the translated discrete exponential map in \ref{fig:holoTransExp} and \eqref{eq:WeierFormula} we obtain a discrete minimal surface in Euclidean space ($\mu = 1$) on the left, a discrete maximal surface in Lorentz space ($\mu = -1$) on the right and a discrete i-minimal surface ($\mu = 0$) in the middle]{%
    \label{fig:ZMCTransExp}
    \includegraphics[width=0.6\textwidth]{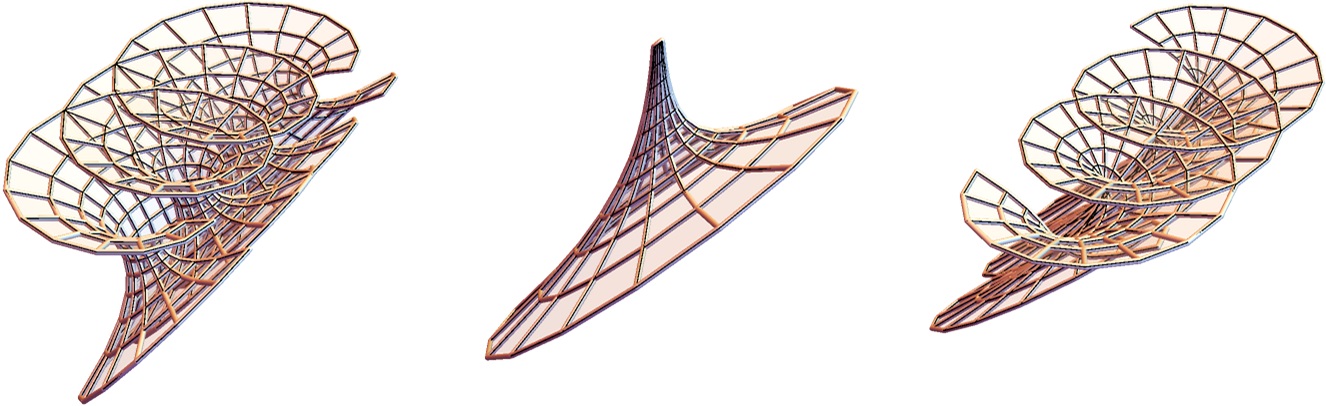}%
	}%
	\caption{Thre discrete holomorphic functions and the respective discrete zero mean curvature surfaces in Euclidean, Lorentz and isotropic space.}%
	\label{fig:ZMC}%
\end{figure}

We obtain explicit formulas for $x$ in terms of the discrete holomorphic function $\phi$ by using the formula for $\zeta$ in \eqref{eq:zetaLift}. For 
\begin{align*}
	\p = \tfrac{1}{2}(1-\mu)e_3 + \tfrac{1}{2}(1+\mu)e_0,
\end{align*}
we obtain 
\begin{align}\label{eq:WeierFormula}
\begin{split}
	dx_{ij} = \Re \left\{\tfrac{1}{m_{ij}(\phi_i-\phi_j)}\left(\tfrac{1-\mu}{2}(\phi_i + \phi_j)e_0 + (1-\mu\phi_i\phi_j)e_1 \right. \right.\\ 
\left.\left. + \i (1+\mu\phi_i\phi_j)e_2 + \tfrac{1+\mu}{2}(\phi_i+\phi_j)e_3\right) \right\},
\end{split}
\end{align}
as representation of a circular discrete surface in $\R^3\cong \spn{\p}^\perp$. For $\mu=1$, that is $\p = e_0$, we recover the usual discrete Weierstrass representation for discrete isothermic minimal surfaces in Euclidean space (see \cite[Thm 9]{bobenko1996}). If we set $\mu=-1$, that is $\p=e_{3}$, we get the Weierstrass representation for discrete maximal surfaces (see \cite[Thm~1.1]{yasumoto2015}). For $\mu=0$, that is $\p = \frac{e_{0}+e_{3}}{2}$, we obtain a discrete surface in isotropic space $I_{\p, \o}$ if we choose the initial point of $x$ to satisfy $(x,\p)=-1$. According to Example \ref{exp:LightlikeCurvature}.(\ref{exp:lightlikeIso}), $x$ is a discrete i-minimal surface as in \cite[Sect 4]{pottmann2007} and we derived a Weierstrass representation formula for this class, namely, 
\begin{align*}
	dx_{ij} = \text{Re}\{( e_{1} + \i e_{2} - (\phi_{i}+\phi_{j})\mathfrak{p}) \tfrac{1}{m_{ij}(\phi_{i}-\phi_{j})}\}. 
\end{align*}
In Figure \ref{fig:ZMC} we used the representation formula \eqref{eq:WeierFormula} and different values of $\mu$ to obtain discrete minimal, i-minimal and maximal surfaces from given holomorphic data (a discrete power function and the discrete exponential map). 

We summarize the results of this subsection in the following theorem: 
\begin{satz}\label{thm:HyperplaneReps}
	The Weierstrass-type representations for 
	\begin{itemize}
		\item discrete minimal surfaces in Euclidean $3$-space, 
		\item discrete maximal surfaces in Lorentzian $3$-space, and
		\item discrete i-minimal surfaces in isotropic $3$-space
	\end{itemize}
	can be viewed as certain applications of the $\Omega$-dual transformation to a prescribed discrete lightlike Gauss map.
\end{satz}

\subsection{Discrete surfaces in quadrics in \texorpdfstring{$\R^{3,1}$}{R31}}\label{sec:Quadric}
For the discrete isothermic surface $G$ given by \eqref{eq:phiLift}, we defined its family of flat connections as
\begin{align*}
	\Gamma(t)_{ij} := \Gamma^{G_i}_{G_j}\left(1-\tfrac{t}{m_{ij}}\right).
\end{align*}
For fixed $m\in\R^G$ (see Subsection \ref{sec:Calapso}), because $\Gamma(m)$ is flat, there (locally) exist parallel sections $x:\Z^2 \to \R^{3,1}$, that is, discrete surfaces such that on each edge $(ij)$
\begin{align*}
	x_j = \Gamma(m)_{ji}x_i.
\end{align*}
For such a surface $(x_i,x_i) = (x_j,x_j)$ because $\Gamma(m)_{ij} \in O(3,1)$, so $(x,x)=const.$ and $x$ takes values in a quadric in $\R^{3,1}$. 

If $x_i \perp G_i$, then one has that $dx_{ij}, dx_{il}\in G_i$ (see \eqref{eq:Gamma}) and the image of the quadrilateral $(ijkl)$ under $x$ degenerates. Away from such points, we may choose $g^n \in \Gamma G$ so that $(g^n, x) =-1$. We will now show that $x$ and $g^n$ are edge parallel and $A(x,g^n)=0$: 

Firstly, 
\begin{align*} 
	(x_{i},g^n_{j}) = (\Gamma(m)_{ij}x_{j},g^n_{j}) = (x_{j},\Gamma(m)_{ji} g^n_{j}) = -\left(1-\tfrac{m}{m_{ij}}\right),
\end{align*}
implying, by \eqref{eq:Gamma}, that we have that 
\begin{align*}
	x_{j} = \Gamma(m)_{ji} x_{i} = x_{i} - \tfrac{m}{m_{ij}(g^n_{i},g^n_{j})}dg^n_{ij}, 
\end{align*}
thus, $dx_{ij} = \tfrac{m}{m_{ij}(g^n_{i}, g^n_{j})}dg^n_{ij}$. This tells us that $x$ is edge-parallel to $g^n$ and that we can write ${\zeta_{ij} = \frac{1}{m} g^n_{ij}\wedge dx_{ij}}$. This implies $A(x,g^n)=0$ and by Proposition \ref{prop:closedForm}, $L:=x+G$ is a discrete $L$-isothermic surface.

Assume $x$ takes values in a quadric $Q_\mu$ of constant curvature, given by $\mu:=(\p, \p)\neq 0$ (see example (\ref{exp:CMC}) in \ref{exp:LightlikeCurvature}). A Gauss map of $x$ is determined by $g^n$ via 
\begin{align*}
	n = \frac{1}{\sqrt{|\mu|}}\left(x + \operatorname{sgn}(\mu)\sqrt{|\mu|} g^n\right),
\end{align*}
and $x$ is a \emph{discrete constant mean curvature (cmc) surface} in the respective space form:
\begin{itemize}
	\item If $\mu<0$, $x$ is cmc $\frac{1}{\sqrt{-\mu}}$ in the hyperbolic space $Q_\mu$. 
	\item If $\mu>0$, $x$ is cmc $\frac{1}{\sqrt{\mu}}$ in de Sitter space $Q_\mu$. 
\end{itemize} 

If $\mu=0$, $x$ is intrinsically flat in the light cone (see example \ref{exp:CMC}.\ref{exp:Lightcone}). The representation we obtain for these discrete surfaces is thus a discrete analogue of the one given in \cite[Sect. 4.2]{pember2020}. 

Conversely, let $x$ be cmc (intrinsically flat) in the quadric $Q_\mu$ (the light cone $\L$) and let $g^n \in \Gamma G$ be a lift of its hyperbolic Gauss map such that $(x,g^n) = -1$. Then $x$ and $g^n$ are edge-parallel and $A(x,g) =0$. As Proposition \ref{prop:closedForm} states the $1$-form $\zeta_{ij} = \tfrac{1}{m}g_{ij}\wedge dx_{ij}$ is closed and by Lemma \ref{lem:MoutardLiftClosed1Form} we have
\begin{align*}
  \frac{1}{m} g_{ij}\wedge dx_{ij} =  \frac{1}{m_{ij}(g_i, g_j)} g_{ij}\wedge dg_{ij},
\end{align*}
which implies $dx_{ij} = \tfrac{m}{m_{ij}(g_i, g_j)}dg_{ij}$. This can be used to show that $(x_i, g_j) = -(1-\tfrac{m}{m_{ij}})$ and subsequently
\begin{align*}
	\Gamma(m)_{ji}x_i = x_j,
\end{align*}
that is, $x$ is a parallel section of $\Gamma(m)$. This proves the following. 

\begin{prop}\label{prop:DarbouxOfGaussMap}
	A discrete surface $x$ in a quadric $Q_\mu$ (the light cone $\L$) is cmc $\tfrac{1}{\sqrt{|\mu|}}$ (intrinsically flat) if and only if it is a parallel section of $\Gamma(m)$ for some $m\in \R^G$, defined by its lightlike Gauss map. 
\end{prop}

\begin{bem}\label{rem:CopiesOfH}
	For $\mu<0$, $Q_\mu$ consists of two copies of hyperbolic space $\H^3$. A net $x$ takes values in $\H^3$ if $(e_0,x)<-1$, that is if its Hermitian representative $X$ satisfies $\operatorname{tr}X >2$. Given an initial point $x_0 \in \H^3$, a discrete cmc $1$ surface in $Q_{-1}$ is formed by applying the hyperbolic rotations $\Gamma(m)$.  We learn: \emph{Every discrete cmc 1 surface in $\H^3$ is constructed by the application of hyperbolic rotations to a seed point.}
\end{bem}

As we have seen in Example \ref{exp:Gamma}, $\Gamma(m)$ takes the following form in the Hermitian model:
\begin{align}\label{eq:GammaRep}
	B(m)_{ij} =\tfrac{1}{(\phi_i - \phi_j)\sqrt{m_{ij}}\sqrt{m_{ij}-m}}\begin{pmatrix}
		m_{ij}(\phi_i-\phi_j) - m\phi_i &m\phi_i\phi_j \\
		-m &m_{ij}(\phi_i - \phi_j)+m\phi_j
	\end{pmatrix}.
\end{align}
Define a map $F(m):\Z^2 \to SL(2,\C)$ by
\begin{align*}
	F(m)_j = F(m)_iB(m)_{ij}
\end{align*}
Thus, $F(m)$ is the Hermitian representative of $T(m)$, as defined via \eqref{eq:Ttrafo}.

Since $x$ is parallel with respect to $\Gamma(m)$, there exists a fixed vector $\p \in \R^{3,1}$ such that
\begin{align*}
	x = T^{-1}(m)\p.
\end{align*}
Now, take 
\begin{align*}
	\p = \tfrac{1}{2}(1+\mu)e_3 + \tfrac{1}{2}(1-\mu)e_0,
\end{align*}
and $x$ is given in the Hermitian model by
\begin{align*}
	X = F(m)^{-1}\begin{pmatrix}
		1 &0 \cr 0 &-\mu
	\end{pmatrix} \left(F(m)^{-1}\right)^\ast.
\end{align*}
 Denoting $\Phi = F(m)^{-1}$, we obtain the following Weierstrass-type representation 
\begin{align}\label{eq:WeierHyp}
	X= \Phi \begin{pmatrix}
		1 &0 \cr 0 &-\mu
	\end{pmatrix} \Phi^\ast,~\textrm{for}~\Phi_j=B(m)_{ij}^{-1}\Phi_i.
\end{align}
Together with Proposition \ref{prop:DarbouxOfGaussMap} this proves the following theorem. 

\begin{satz} \label{thm:CMCRep}
  Let $\phi:\Z^2 \to \C$ be a discrete holomorphic function. Then the discrete surface represented by $X$ in \eqref{eq:WeierHyp} has constant discrete mean curvature $H = \tfrac{1}{\sqrt{|\mu|}}$ (is intrinsically flat) in the quadric $Q_\mu$ (the light cone $\L$ if $\mu = 0$). Every discrete cmc $\tfrac{1}{\sqrt{|\mu|}}$ (intrinsically flat) surface in $Q_\mu$ ($\L$) can locally be constructed this way. 
  
  Moreover, these Weierstrass-type representations can be viewed as certain applications of the $\Omega$-dual transformation to the prescribed discrete lightlike Gauss map
  \begin{align*}
    G=\spn{\begin{pmatrix}
      |\phi|^2 &\phi \cr \overline{\phi} &1
    \end{pmatrix}}.
  \end{align*}
\end{satz}

In Figure \ref{fig:cmc1H3pair} there are two examples of discrete cmc 1 surfaces in $\H^3$, obtained via \eqref{eq:WeierHyp}, pictured in the Poincar\'e ball model. 

\begin{figure}%
	\centering
	\subfigure[][A discrete cmc $1$ surface in hyperbolic space, obtained from the discrete power function \ref{fig:holoPower} via \eqref{eq:WeierHyp}]{%
		\label{fig:cmc1H3power}
		\includegraphics[width=0.4\textwidth]{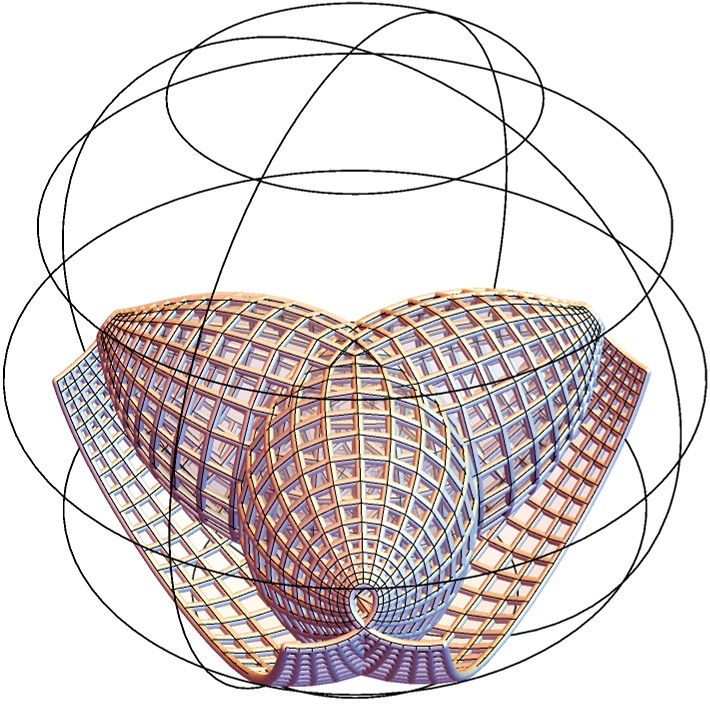}%
	}%
  \hspace{1cm}
		\subfigure[][A discrete cmc $1$ surface of revolution in hyperbolic space, obtained from the discrete exponential map via \eqref{eq:WeierHyp} with a particular choice of initial condition]{%
		\label{fig:cmc1H3rev}%
		\includegraphics[width=0.4\textwidth]{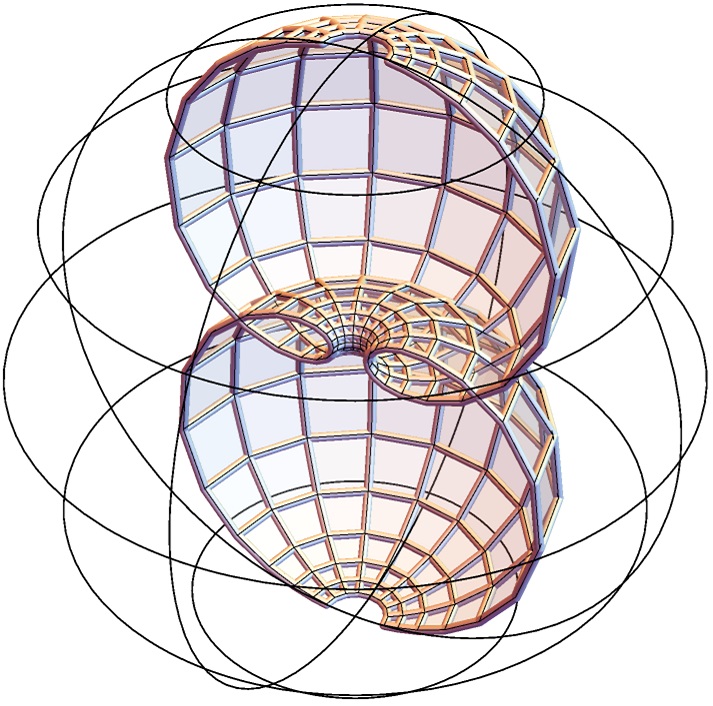}%
	}%
  \caption{Two examples of discrete cmc $1$ surfaces in $\H^3$ shown in the Poincar\'e ball model}%
	\label{fig:cmc1H3pair}%
\end{figure}

In the cases of discrete cmc 1 surfaces in $\H^3$ our construction coincides with the M\"{o}bius geometric construction of \cite[Section 5]{bobenko2014}, where it was also noted that these nets in $\H^3$ are obtained as Darboux transforms of their hyperbolic Gauss maps. However, this Weierstrass-type representation does not coincide with the one given in \cite{hoffmann2012}. Using the Lawson correspondence, we will bridge the gap between the two representations in the next subsection.

\subsection{The Lawson correspondence}\label{sec:Lawson}
Let $x=T(m)^{-1}\p$ be a discrete surface of vanishing light cone mean curvature, constructed as a parallel section of the flat connection $\Gamma(m)$ of its discrete lightlike Gauss map as in the last section. Using \eqref{eq:CalapsoX} the Calapso transform $x(m)$ then satisfies
\begin{align*}
	dx(m)_{ij} &= \frac{m}{(g_i,g_j)m_{ij}}\left((g(m)_i,\p)g(m)_j - (g(m)_j,\p)g(m)_i)\right) \\
	&= -m\zeta(m)_{ij}\p.
\end{align*} 
Hence, $x(m)$ is of vanishing mean curvature in a hyperplane according to Subsection \ref{sec:Hyperplane}. We thus have proved the following Laguerre geometric interpretation of the Lawson correspondence noted in \cite[Lem 4.2]{hertrich-jeromin2000} and \cite[Sect 5.5]{burstall2014} which was generalized in \cite[Sec 4.2]{burstall2018}.

\begin{prop}\label{prop:Lawson}
The Calapso transformation of discrete L-isothermic surfaces perturbs
  \begin{itemize}
    \item discrete cmc $1$ surfaces in $\H^3$ and discrete minimal surfaces in $\R^3$, 
    \item discrete cmc $1$ surfaces in $\S^{2,1}$ and discrete maximal surfaces in $\R^{2,1}$ and
    \item discrete intrinsically flat surfaces in the light cone $\L$ and discrete i-minimal surfaces in isotropic $3$-space.
  \end{itemize}
\end{prop}

We will use this to bridge the gap between the representation \eqref{eq:WeierHyp} for discrete cmc $1$ surfaces in $\H^3$ and the one given in \cite{hoffmann2012}: 

For $x:\Z^2\to \H^3$ with hyperbolic Gauss map $G$, we call $\widehat{G}=T(m) \cdot G$ the \emph{discrete secondary Gau{ss} map of $x$}. As $\widehat{G}$ is conformally equivalent to $G$, it is isothermic and for a suitable holomorphic map $\psi: \Sigma \to \C$ (obtained via the inverse of \eqref{eq:StereoProj}) we can write 
\begin{align*}
	\widehat{G} = \spn{ \begin{pmatrix}
		|\psi|^2 &\psi \\ \overline{\psi} &1
	\end{pmatrix}}.
\end{align*} 

By Proposition \ref{lem:Additive}, we have $T(m)^{-1} = T^m(-m)$. Hence, the discrete cmc surface $x$ in $\H^3$ satisfies
\begin{align}\label{eq:PreRep2}
	x = T^{-1}(m)\p = T^{m}(-m)\p, 
\end{align}
with $\p = e_0$. Denoting the Hermitian representative of $T^m(-m)$ by $\Psi$ we obtain the following representation
\begin{align}\label{eq:Rep2}
	X=\Psi \Psi^{\ast},~\textrm{with}~\Psi_j = \Psi_i B^m(-m)_{ij},
\end{align}
where $B^m(-m)$ denotes the Hermitian representation of $\Gamma^m(-m)$ as given in \eqref{eq:GammaRep}. 

This corresponds to the representation given in \cite{hoffmann2012}: There, a discrete cmc $1$ surface in $\H^3$ is given via
\begin{align*}
	Y=\frac{1}{\det E}EE^\ast, 
\end{align*}
where $E$ satisfies
\begin{align*}
	E_j - E_i = E_i \begin{pmatrix}
		\psi_i &-\psi_i\psi_j \\ 1 &-\psi_j
	\end{pmatrix}\frac{\lambda}{m_{ij}(\psi_j-\psi_i)},
\end{align*}
that is
\begin{align*}
	E_j = E_i \begin{pmatrix}
		m_{ij}(\psi_i-\psi_j)-\lambda\psi_i &\lambda\psi_i\psi_j \\
		-\lambda &m_{ij}(\psi_i-\psi_j) + \lambda\psi_j
	\end{pmatrix}\frac{1}{m_{ij}(\psi_i-\psi_j)}
\end{align*}
It is readily seen that
\begin{align*}
	\det E_j = \frac{m_{ij}-\lambda}{m_{ij}}\det E_i,
\end{align*}
and by setting $\Psi = \frac{1}{\sqrt{\det E}}E$ and $\lambda = -m$, we recover \eqref{eq:Rep2}. 

\begin{figure}%
	\centering
	\subfigure[][The discrete power function $z\mapsto z^{2/3}$ is the hyperbolic Gauss map of the surface shown in \ref{fig:hypGmap1Cmc1H3} and the secondary Gauss map of the surface shown in \ref{fig:holo1Cmc1H3}]{%
		\label{fig:holo1all}
		\includegraphics[width=0.4\textwidth]{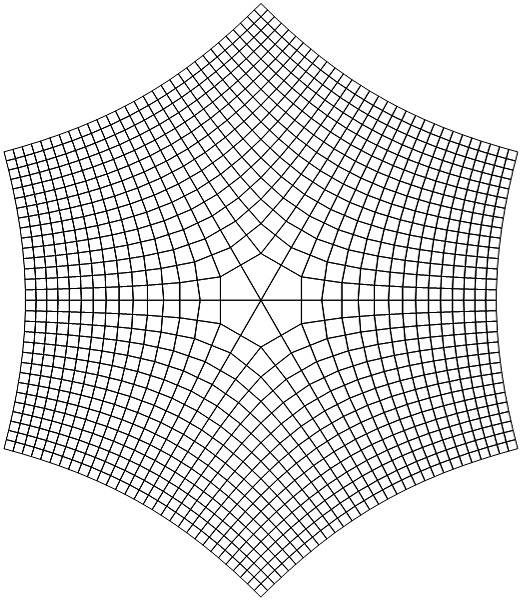}%
	}%
  \hspace{1cm}
  	\subfigure[][The discrete cmc $1$ surface obtained (via \eqref{eq:Rep2}) from the discrete power function \ref{fig:holo1all} which is also its secondary Gauss map. Its holomorphic Gauss map is shown in \ref{fig:hypGmap1all}.]{%
		\label{fig:holo1Cmc1H3}
		\includegraphics[width=0.4\textwidth]{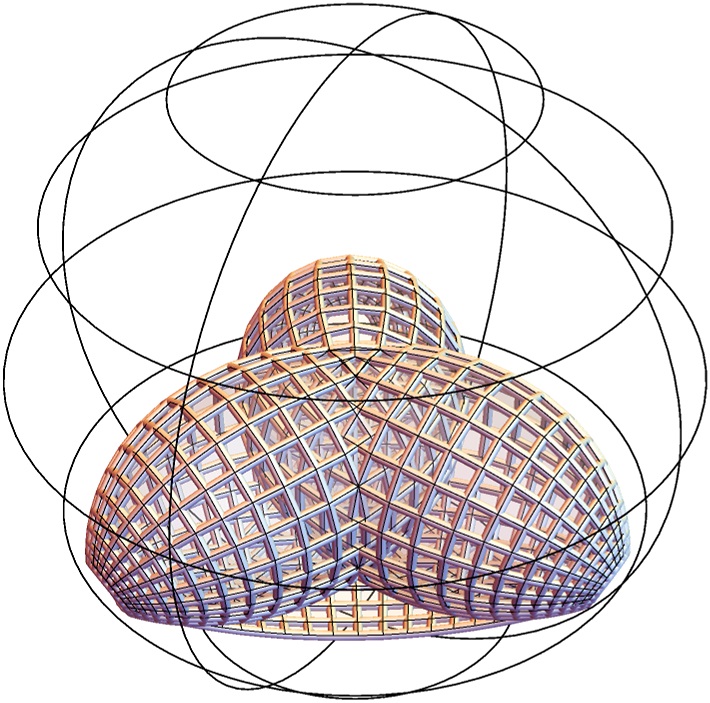}%
	}%
  \\
		\subfigure[][The hyperboic Gauss map of the surface shown in \ref{fig:holo1Cmc1H3} is used as data to obtain the surface in \ref{fig:hypGmap1Cmc1H3}]{%
		\label{fig:hypGmap1all}%
		\includegraphics[width=0.4\textwidth]{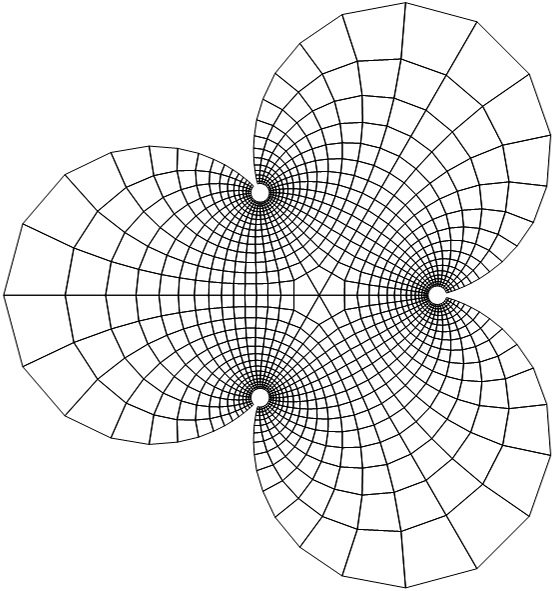}%
	}%
  \hspace{1cm}
		\subfigure[][The dual surface of the one shown in \ref{fig:holo1Cmc1H3}, obtained from the discrete holomorphic map in \ref{fig:hypGmap1all} via \eqref{eq:Rep2}.]{%
		\label{fig:hypGmap1Cmc1H3}%
		\includegraphics[width=0.4\textwidth]{Pics/others/hypGmap1cmc1H3.jpg}%
	}%
  \caption{A discrete cmc $1$ surface in $\H^3$  and its dual surface in the Poincar\'e ball model. The discrete holomorphic functions can be regarded of the hyperbolic and secondary Gauss maps of the surfaces.}%
	\label{fig:dualSurfs}%
\end{figure}

\begin{bem}\label{rem:Dual}
	The hyperbolic Gauss map $G$ of $x$ is given by $G=T^m(-m)\cdot \widehat{G}$. Hence $G$ is the projection of the discrete function
	\begin{align}\label{eq:DualHolo}
		\phi := \frac{a\psi + b}{c\psi + d},~\textrm{where } \Psi =\begin{pmatrix}
        a &b \cr c &d
      \end{pmatrix} 
	\end{align}
	Then, we can directly show that $E_i^{-1}$ satisfies
	\begin{align*}
		E_j^{-1} - E_i^{-1} = E_i^{-1} \begin{pmatrix}
			\phi_i &-\phi_i\phi_j \cr 1 &-\phi_j
		\end{pmatrix}\frac{-\lambda}{(m_{ij}-\lambda)(\phi_j - \phi_i)}
	\end{align*}
	which is the discrete analogue of the duality formula described in \cite[Prop 4]{umehara1997}. Then, $E^{-1}$ defines the \emph{discrete dual cmc 1 surface} of $x$ via its Hermitian representative
	\begin{align*}
		X^\# = (\operatorname{det}E)E^{-1} \left(E^{-1}\right)^\ast = \Psi^{-1} \left(\Psi^{-1}\right)^\ast.
	\end{align*}
	Since $\Psi^{-1}_j = B(m)_{ij} \Psi^{-1}_i$, $X^\#$ is obtained via \eqref{eq:WeierHyp}. The hyperbolic Gauss map of $X^\#$ is given by $\widehat{G}$. In Figure \ref{fig:dualSurfs} we see a pair of dual cmc 1 surfaces in $\H^3$.
\end{bem}

Note that representations like \eqref{eq:Rep2} for other space forms are easily obtained by applying \eqref{eq:PreRep2} to ${\p = \tfrac{1}{2}(1+\mu)e_3 + \tfrac{1}{2}(1-\mu)e_0}$ for different $\mu$. For $\mu=0$ this yields a Weierstrass representation for discrete intrinsically flat surfaces in the light cone $\L$. Because of the duality with \eqref{eq:WeierHyp} mentioned in Remark \ref{rem:Dual}, we learn from Theorem \ref{thm:CMCRep} that every discrete surface with constant mean curvature $\frac{1}{\sqrt{|\mu|}}$ in the quadric $Q_\mu$ and every intrinsically flat surface in $\L$ can be constructed in this way.

\begin{satz}\label{thm:QuadricReps}
  Let $\psi: \Z^2 \to \C$ be a discrete holomorphic function. Then the discrete surface represented by 
  \begin{align*}
    X= \Psi \begin{pmatrix}
      1 &0 \cr 0 &-\mu
    \end{pmatrix} \Psi^\ast,~\textrm{with}~\Psi_j = \Psi_i B^m(-m)_{ij},
  \end{align*}
  has constant discrete mean curvature $H = \tfrac{1}{\sqrt{|\mu|}}$ (is intrinsically flat) in the quadric $Q_\mu$ (the light cone $\L$ if $\mu = 0$). Every discrete cmc $\tfrac{1}{\sqrt{|\mu|}}$ (intrinsically flat) surface in $Q_\mu$ ($\L$) can locally be constructed this way. 
  
  Moreover, these Weierstrass-type representations can be viewed as certain applications of the $\Omega$-dual transformation to the prescribed discrete lightlike Gauss map
  \begin{align*}
    G=\spn{\begin{pmatrix}
      |\phi|^2 &\phi \cr \overline{\phi} &1
    \end{pmatrix}},
  \end{align*}
  where $\phi$ is given by \eqref{eq:DualHolo}.
\end{satz}

\subsection{Discrete linear Weingarten surfaces of Bryant and Bianchi type}\label{sec:Bryant}
In this subsection we employ the Lawson correspondence described in Subsection \ref{sec:Lawson} to recover the Weierstrass representation of discrete linear Weingarten surfaces of Bryant and Bianchi type (abbreviated as BrLW and BiLW, respectively) given in \cite{yasumoto2018} (see \cite{burstall2018} for a geometric interpretation in the realm of Lie sphere geometry). Thus, let $G=\spn{g}$ be a discrete isothermic surface in the projectified light cone with flat connection $\Gamma(m)$ and define $T(m)$ as in \eqref{eq:Ttrafo}. Choose a vector $\p$ with $(\p,\p) = \mu$. Then $x^M=T^{-1}(m)\p$ is a discrete cmc $\tfrac{1}{\sqrt{|\mu|}}$ surface in $Q_\mu$ (for $\mu\neq 0$) or intrinsically flat in the light cone (for $\mu = 0$). Now define for an arbitrary lift $g\in \Gamma G$
\begin{align*}
	x^\pm:=x^M - \frac{\mu\pm 1}{2(g, x^M)}g.
\end{align*}
Clearly, $L=x^\pm + G = x^M + G$ is discrete L-isothermic and since 
\begin{align*}
	(x^\pm, x^\pm) = \mp 1,
\end{align*}
$x^\pm$ takes values in hyperbolic space $\H^3$ or de Sitter space $\S^{2,1}$. $G$ is then a discrete hyperbolic Gauss map of $x^{\pm}$ and induces the Gauss maps $n^\pm$ of $x^\pm$ via
\begin{align*}
  n^\pm=\mp \tfrac{g}{(g,x^\pm)} - x^\pm.
\end{align*}

\begin{lem}\label{lem:Bryant}
	The discrete surface $x^\pm$ is BrLW/BiLW in hyperbolic/de Sitter space with respect to the Gauss map induced by $G$, that is,
	\begin{align}
		(\mu \pm 1)K - 2\mu H + (\mu \mp 1)=0,
	\end{align}
	where $H$ and $K$ denote the discrete mean and discrete extrinsic Gauss curvature with respect to $n^\pm$, as defined in \eqref{eq:MeanCurvature} and \eqref{eq:GaussCurvature} respectively. \\
	Conversely, for every discrete BrLW/BiLW surface $x^\pm$, the discrete surface defined by 
	\begin{align*}
		x^M = x^\pm + \frac{\mu\pm 1}{2(g, x^\pm)}g
	\end{align*}
	is cmc $\tfrac{1}{\sqrt{|\mu|}}$ in $Q_\mu$ (for $\mu \neq 0$) or intrinsically flat in the light cone (for $\mu =0$).
\end{lem}

\begin{proof}
	It is straightforward to see that
	\begin{align*}
		A(x^M, g) &= (g,x^\pm) A\left(-\frac{\mu \pm 1}{2}n^\pm - \frac{\mu \mp 1}{2}x^\pm, -n^\pm-x^\pm\right)\\
		&=\frac{A(x^\pm,x^\pm)(g,x^\pm)}{2}\left\{
			(\mu \pm 1)K-2\mu H +(\mu \mp 1)
		\right\},
	\end{align*}
	proving that $x^\pm$ is BrLW/BiLW if and only if $x^M$ has vanishing light cone mean curvature. This concludes the proof.
\end{proof}

\begin{bem}\label{rem:BryantBiachniDual}
	The two maps $x^\pm$ are orthogonal, hence, form a symmetric pair of discrete surface and Gau{ss} map. Lemma \ref{lem:Bryant} therefore proves that the Gau{ss} map of a discrete BrLW surfacee is BiLW (and vice versa). Of course, $n$ is only defined up to its sign, which we chose in alignment with \cite{yasumoto2018}. Note that
	\begin{align*}
		n^+ = -x^-, 
	\end{align*}
	that is, for each discrete BrLW surface $x^+$ in $\H^3$, the Gauss map $n^+$ and the discrete surface $x^-$ are antipodal discrete BiLW surfaces in $\S^{2,1}$.
\end{bem}

\begin{bsp}
In the case that $\mu = 0$, we have that $x^{\pm}$ are discrete flat fronts and $x^{M}$ is a map into the light cone. In \cite{hertrich-jeromin2021a} it is shown that discrete flat fronts in $\mathbb{H}^{3}$ are the orthogonal nets of the cyclic congruence associated to a Darboux pair of maps into the two sphere. We recover this Darboux pair as $(G,\hat{G})$, where $G$ is the discrete lightlike Gauss map of $x^+$ and $\hat{G}:=\spn{x^{M}}$. 
\end{bsp}

In the Hermitian model, $x^\pm$ corresponds to 
\begin{align*}
	X^\pm = X^M - \frac{\mu \pm 1}{2(g, x^M)}g  = \Psi\cdot \left(\p - \frac{\mu\pm 1}{2(g, x^M)}\Psi^{-1}\cdot g\right),
\end{align*}
where $X^M$ denotes the Hermitian representative of $x^M$. The discrete secondary Gauss map of $x^\pm$, namely ${\widehat{G}=\Psi^{-1}\cdot G = \spn{\hat{g}}}$, is obtained via inverse stereographic projection from a discrete holomorphic map ${\psi:\Z^2 \to \C}$ as
\begin{align*}
	\widehat{g} = 2 \begin{pmatrix}
		|\psi|^2 &\psi \\ \overline{\psi} &1
	\end{pmatrix}.
\end{align*}
If we choose 
\begin{align*}
\p = \tfrac{1}{2}(1-\mu)e_0 + \tfrac{1}{2}(1+\mu)e_3 =\begin{pmatrix}
	1 &0 \cr 0 &-\mu
\end{pmatrix},
\end{align*}
we have 
\begin{align*}
	(g, x^M) = (\Psi\cdot \hat{g}, \Psi\p) = (\hat{g}, \p) = \mu|\psi|^2 - 1,
\end{align*}
and thus
\begin{align}\label{eq:BryantRep}
	\begin{split}
	X^\pm &= \Psi\cdot C^\pm \\ \textrm{with}\quad 
	C^\pm &= \begin{pmatrix}
		1 &0 \\ 0 &-\mu
	\end{pmatrix} + \frac{\mu\pm 1}{1-\mu|\psi|^2} \begin{pmatrix}
		|\psi|^2 &\psi \\ \overline{\psi} &1
	\end{pmatrix} = \frac{1}{1-\mu|\psi|^2} \begin{pmatrix}
		1\pm |\psi|^2 &(\mu\pm 1)\psi \\ (\mu \pm 1)\overline{\psi} &\pm 1 + \mu^2|\psi|^2
	\end{pmatrix}.
	\end{split}
\end{align}

\begin{satz}\label{thm:AllBryant}
  Let $\psi: \Z^2 \to \C$ be a discrete holomorphic function. Then the discrete surfaces represented by \eqref{eq:BryantRep} are a pair of discrete BrLW/BiLW surfaces in hyperbolic/de Sitter space. 
  
  Every pair of discrete BrLW/BiLW surfaces can be constructed in this way.
\end{satz}

\begin{proof}
	The only thing left to prove is that every discrete BrLW/BiLW surface arises via the Weierstrass-type representation \eqref{eq:BryantRep}. Since this is just an algebraic transformation of \eqref{eq:Rep2}, Lemma \ref{lem:Bryant} and Theorem \ref{thm:QuadricReps} show that it produces all discrete BrLW/BiLW surfaces.
\end{proof}

We see this representation in Figure \ref{fig:BryantBianchiSurfs} where discrete holomorphic functions are used to obtain pairs of discrete BrLW/BiLW surfaces in hyperbolic and de Sitter space. 

\begin{figure}%
	\centering
	\subfigure[][A discrete cmc $1$ surface in $\H^3$ (left, see \ref{fig:cmc1H3power}) and a discrete surface of constant harmonic mean curvature $1$ in $\S^{2,1}$ (right)]{%
		\label{fig:CMC1H3HMC1S21}
		\includegraphics[width=0.4\textwidth]{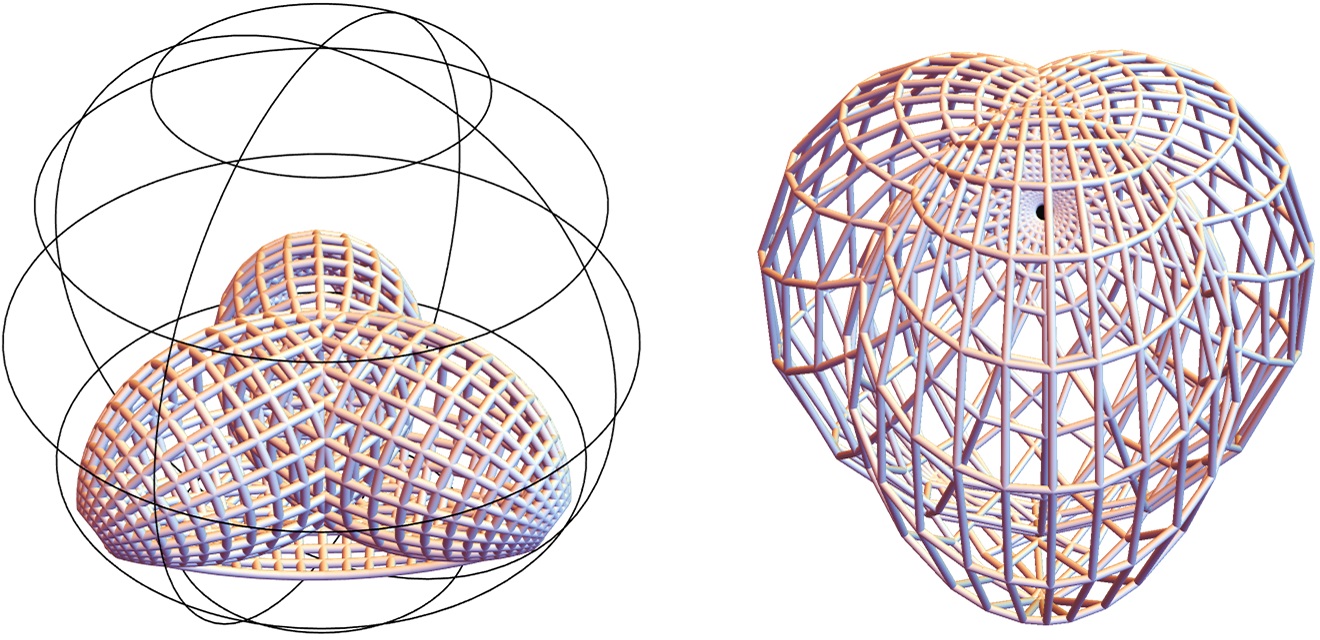}%
  }%
  \hspace{1cm}
  \subfigure[][A discrete cmc $1$ surface in $\S^{2,1}$ (right) and a discrete surface of constant harmonic mean curvature $1$ in $\H^3$ (left)]{%
		\label{fig:HMC1H3CMC1S21}
		\includegraphics[width=0.4\textwidth]{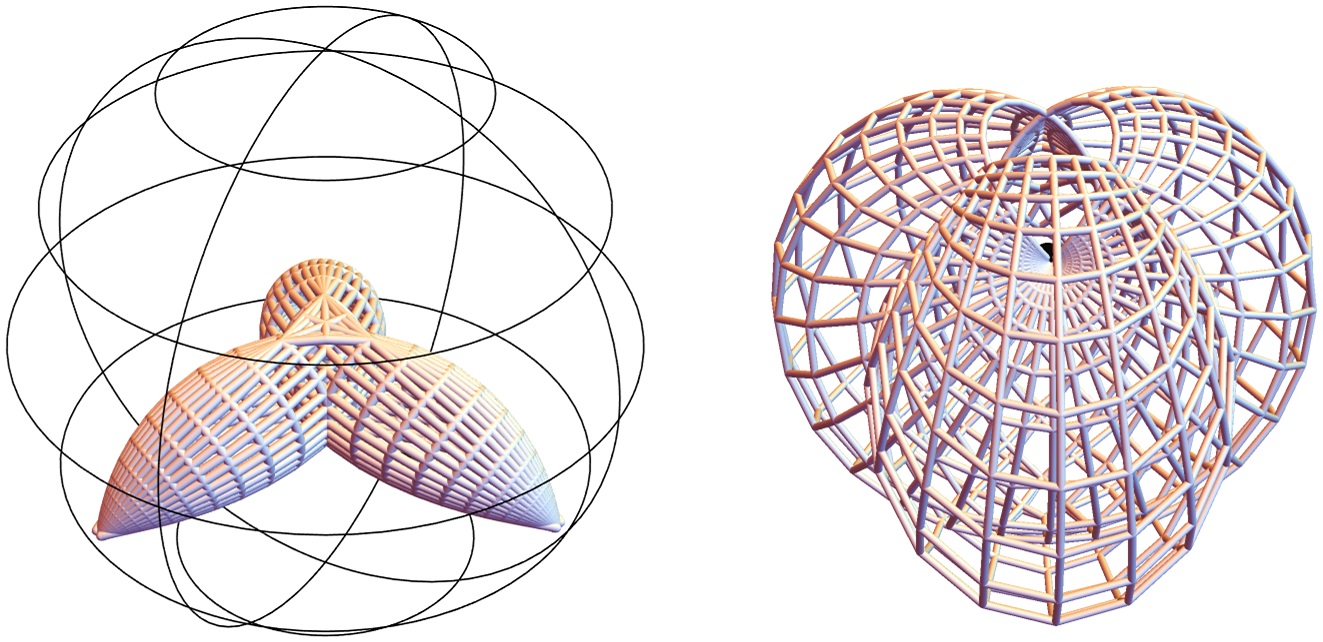}%
	}%
  \\
  	\subfigure[][Two discrete intrinsically flat surfaces in $\H^3$ (left) and $\S^{2,1}$ (right)]{%
		\label{fig:flatH3flatS21}
		\includegraphics[width=0.4\textwidth]{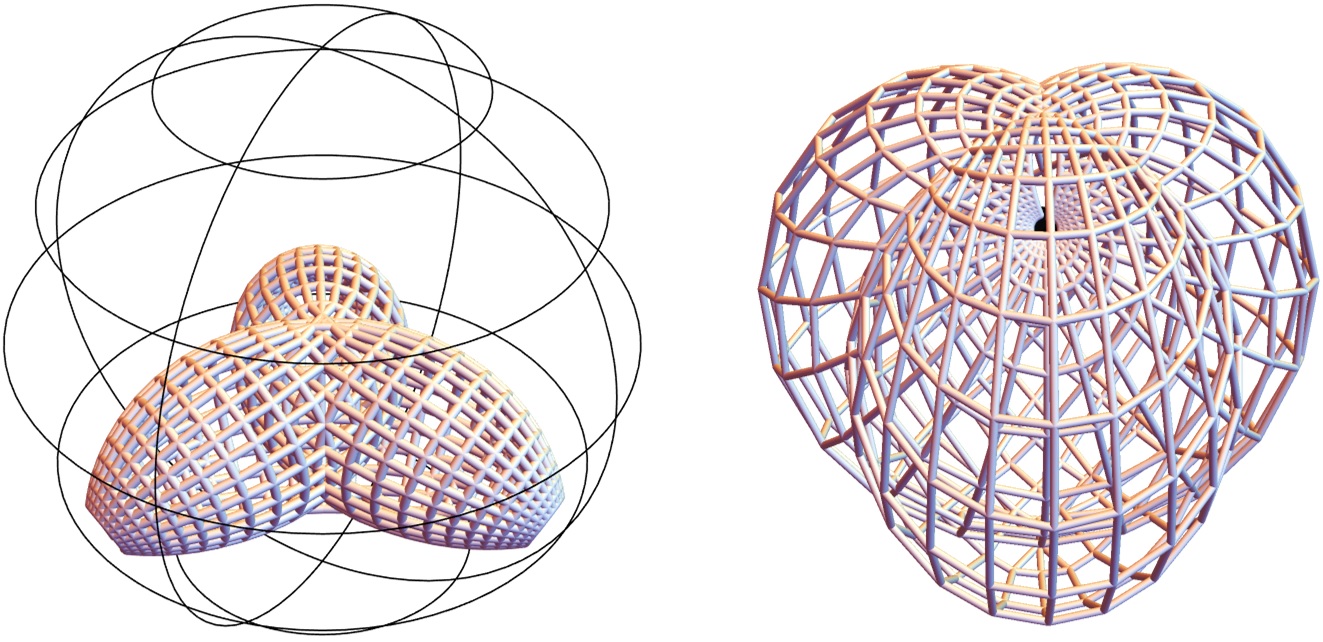}%
	}%
  \caption{Examples of discrete BrLW/BiLW surfaces in $\H^3$ or $\S^{2,1}$ which form a pair of discrete surface and hyperbolic Gauss map (see Remark \ref{rem:BryantBiachniDual}). All examples are created using the representation \eqref{eq:BryantRep} with the discrete power function as data. We show surfaces in hyperbolic space in the Poincar\'e ball model and surfaces in de Sitter space in the hollow ball model.}%
	\label{fig:BryantBianchiSurfs}%
\end{figure}

To see that we have recovered the representation of linear Weingarten surfaces given in \cite[Prop 5.1]{yasumoto2018}, recall the Weierstrass representation formulas for discrete BrLW and BiLW surfaces, as given there: Choose $t$ and $\lambda$ so that
\begin{align*}
	\mathcal{T}_i:=1+t|\psi_i|^2 \neq 0,\quad\textrm{and}\quad 1-\frac{\lambda}{m_{ij}}\neq 0.
\end{align*}
We define $E: \Z^2 \to SL(2,\C)$ by integration of
\begin{align*}
	E_i^{-1}E_j=\frac{1}{\sqrt{1-{\lambda}{m_{ij}}}} \begin{pmatrix}
		1 &\psi_j-\psi_i \\ \frac{\lambda}{m_{ij}(\psi_j-\psi_i)} &1 
	\end{pmatrix},
\end{align*}
and another matrix-valued map 
\begin{align*}
	L_i=\begin{pmatrix}
		0 &\sqrt{\mathcal{T}_i} \\ -\frac{1}{\sqrt{\mathcal{T}_i}} &-\frac{t\overline{\psi}_i}{\sqrt{\mathcal{T}_i}} 
	\end{pmatrix}.
\end{align*}
Then, we may define two discrete surfaces $(f^+, f^-): \Z^2 \to \H^3 \times \S^{2,1}$ by
\begin{align*}
	f_i^\pm = \operatorname{sgn}(\mathcal{T}_i)E_iL_i \begin{pmatrix}
		1 &0 \\ 0 &\pm 1
	\end{pmatrix}\left( E_iL_i \right)^\ast,
\end{align*}
which are BrLW and BiLW respectively.

The map $\widetilde{\Psi}$, defined by
\begin{align*}
	\widetilde{\Psi} = E \begin{pmatrix}
		1 &-\psi \\0 &1
	\end{pmatrix},
\end{align*}
satisfies
\begin{align*}
	\widetilde{\Psi}_i^{-1}\widetilde{\Psi}_j =\begin{pmatrix}
		1 &\psi_i \cr 0 &1
	\end{pmatrix}\left( E_i^{-1}E_j\right) \begin{pmatrix}
		1 &-\psi_j \cr 0 &1
	\end{pmatrix} = B^m(\lambda)_{ij}
\end{align*}
which is the system \eqref{eq:Rep2} for $\lambda = -m$. Hence, under the right initial conditions, $\Psi = \widetilde{\Psi}$ and the discrete surfaces $f^\pm$ are given as

\begin{align*}
	f^\pm=\operatorname{sgn}(\mathcal{T})\Psi \cdot \left[
		\begin{pmatrix}
		1 &\psi \\0 &1
	\end{pmatrix}L
		\right] \begin{pmatrix}
		1 &0 \\0 &\pm 1
	\end{pmatrix}\left[
		\begin{pmatrix}
		1 &\psi \\0 &1
	\end{pmatrix}L
		\right]^\ast.
\end{align*}
We compute
\begin{align*}
	\left[
		\begin{pmatrix}
		1 &\psi \cr 0 &1
	\end{pmatrix}L
		\right] \begin{pmatrix}
		1 &0 \cr 0 &\pm 1
	\end{pmatrix}\left[
		\begin{pmatrix}
		1 &\psi \cr 0 &1
	\end{pmatrix}L
		\right]^\ast =\frac{1}{|\mathcal{T}|} \begin{pmatrix}
		|\psi|^2\pm 1 &\psi(1 \mp t) \cr \overline{\psi}(1 \mp t) &1\pm t^2|\psi|^2
	\end{pmatrix},
\end{align*}
whereby
\begin{align*}
	f^\pm = \Psi \cdot \frac{1}{1+t|\psi|^2}\begin{pmatrix}
			|\psi|^2\pm 1 &\psi(1\mp t) \\ \overline{\psi}(1\mp t) &1\pm t^2|\psi|^2
		\end{pmatrix},
\end{align*}
which coincides with \eqref{eq:BryantRep} under $t=-\mu$ up to the sign ambiguity mentioned in Remark \ref{rem:BryantBiachniDual}.

We summarize this result in the following theorem:
\begin{satz}\label{thm:RepsBianchiBryant}
	The Weierstrass-type representations for discrete BrLW surfaces in hyperbolic space $\H^3$ and discrete BiLW surfaces in de Sitter space $\S^{2,1}$ \cite[Prop 5.1]{yasumoto2018} can be viewed as certain applications of the $\Omega$-dual transformation to a prescribed discrete lightlike Gauss map.
\end{satz}

\section*{Acknowledgements} 
The authors would like to thank Joseph Cho for fruitful discussions during an impromptu stay in Kobe that sparked many ideas in this paper. We also express our gratitude to Udo Hertrich-Jeromin for his input which added significant results.
Furthermore, we gratefully acknowledge financial support from the FWF research project P28427-N35 "Non-rigidity and symmetry breaking" and the JSPS/FWF Joint Project I3809-N32 "Geometric shape generation". The first author was also supported by the GNSAGA of INdAM and the MIUR grant "Dipartimenti di Eccellenza" 2018--2022, CUP: E11G18000350001, DISMA, Politecnico di Torino and gratefully acknowledges support from the JSPS Grant-in-Aid for JSPS fellows 19J10679. The third author was partly supported by JSPS KAKENHI Grant Numbers JP18H04489, JP19J02034, JP20K14314, JP20H01801, JP20K03585, JST CREST Grant Number JPMJCR1911, and Osaka City University Advanced Mathematical Institute (MEXT Joint Usage/Research Center on Mathematics and Theoretical Physics JPMXP0619217849).

\bibliographystyle{plain}
\bibliography{DWfinal}

\vspace{12pt}

\begin{minipage}[t][3cm][c]{0.6\columnwidth}
  \textbf{Mason Pember}\\
  Department of Mathematical Sciences \\
  University of Bath \\ 
  Bath NA2 7AY, UK \\
  \url{mason.j.w.pember@bath.edu}
\end{minipage}

\begin{minipage}[t][3cm][c]{0.6\columnwidth}
  \textbf{Denis Polly}\\
  Department of Mathematics \\
  Kobe University \\
  1-1 Rokkodai-cho\\
  Kobe 657-8501, Japan\\
  \url{denis@math.kobe-u.ac.jp}
\end{minipage}

\begin{minipage}[t][3cm][c]{0.6\columnwidth}
  \textbf{Masashi Yasumoto}\\
  Graduate School of Technology, Industrial and Social Sciences\\
  Tokushima University\\
  2-1 Minamijyousanjima-cho \\
  Tokushima 770-8506, Japan \\
  \url{yasumoto.masashi@tokushima-u.ac.jp}
\end{minipage}

\end{document}